\documentclass[12pt]{article}
\usepackage{amssymb,amsmath,amsthm,mathrsfs}
\usepackage{float}
\usepackage{multirow}
\usepackage{url}
\usepackage{bm}
\usepackage{graphicx}
\usepackage{algorithm,algorithmic,subfigure}
\usepackage{subeqnarray}
\usepackage{cases}
\usepackage{authblk}
\usepackage{blindtext}
\usepackage{xcolor}
\usepackage{cleveref}

\newcommand{\st}{{\textrm{s.t.}}}

\newcommand{\la}{{\langle}}
\newcommand{\ra}{{\rangle}}
\newcommand{\Acal}{{\mathcal{A}}}

\newcommand{\Ical}{{\mathcal{I}}}

\newcommand{\Scal}{{\mathcal{S}}}
\newcommand{\Ucal}{{\mathcal{U}}}

\newcommand{\bx}{\bm{x}}

\newcommand{\bw}{\bm{w}}
\newcommand{\beps}{\bm{\epsilon}}
\newcommand{\Nb}{\mathbb{N}}

\newtheorem{theorem}{Theorem}[section]
\newtheorem{corollary}[theorem]{Corollary}
\newtheorem{lemma}[theorem]{Lemma}

\newtheorem{proposition}[theorem]{Proposition}
 \theoremstyle{definition}
\newtheorem{definition}[theorem]{Definition}
\theoremstyle{remark}

\newtheorem{example}{Example}[section]
\providecommand{\keywords}[1]
{
	\small
	\textbf{\textit{Keywords---}} #1
}

\begin{document}
\date{}

\title{Towards an efficient approach for the nonconvex $\ell_p$-ball projection: algorithm and analysis}

\author[1,2]{Xiangyu~Yang}
\author[3]{Jiashan~Wang}
\author[1]{Hao~Wang*}
\affil[1]{School of Information Science and Technology, ShanghaiTech University, Shanghai 201210, China}
\affil[2]{Center for Applied Mathematics of Henan Province, Henan University, Zhengzhou, 450046, China}
\affil[3]{Department of Mathematics, University of Washington, Seattle, USA}
\affil[1,2]{\textit {\{yangxy3,wanghao1\}@shanghaitech.edu.cn}}
\affil[3]{\textit {jsw1119@gmail.com}}
\maketitle

\begin{abstract}
This paper primarily focuses on computing the  Euclidean projection of a vector onto the $\ell_{p}$ ball in which $p\in(0,1)$. Such a problem emerges as the core building block in statistical machine learning and signal processing tasks because of its ability to promote the sparsity of the desired solution. However, efficient numerical algorithms for finding the projections are still not available, particularly in large-scale optimization. To meet this challenge, we first derive the first-order necessary optimality conditions of this problem. Based on this characterization, we develop a novel numerical approach for computing the stationary point by solving a sequence of projections onto the reweighted $\ell_{1}$-balls. This method is practically simple to implement and computationally efficient. Moreover, the proposed algorithm is shown to converge uniquely under mild conditions and has a worst-case $O(1/\sqrt{k})$ convergence rate. Numerical experiments demonstrate the efficiency of our proposed algorithm. 
\end{abstract}



\section{Introduction}\label{sec_Intro}
In this paper, we are principally interested in computing the Euclidean projections of a point onto the $\ell_{p}$ ball with $p \in (0,1)$, which can be formulated as the following optimization problem
\begin{equation}\label{projection.lp}
	\min\limits_{\bm{x}\in\mathbb{R}^{n}}~ \frac{1}{2}\|\bm{x}-\bm{y}\|_2^2 \quad 
	\textrm{s.t.}~ \Vert\bm{x}\Vert_{p}^{p} \leq \gamma,
\end{equation}
where $\bm{y}\in\mathbb{R}^{n}$ is given,   $\Vert\bm{x}\Vert_{p}^{p} := \sum_{i=1}^{n}\vert x_i\vert^{p}$ and $\gamma > 0$ is a prescribed scalar and is  referred  to as  the radius of $\ell_{p}$ ball. Throughout this paper, $\bm{y}$ is assumed to be outside the $\ell_{p}$ ball, that is, $\|\bm{y}\|_{p}^{p} > \gamma$; otherwise the computation of such a projection is trivial. Notice that~\eqref{projection.lp} may have multiple minimizers. 

Problems of the form \eqref{projection.lp} arise in inverse problems with a broad spectrum of applications, including image denoising in signal processing, computer vision, and statistical machine learning. In general, the imposed $\ell_{p}$ ball constraint can significantly promote sparsity,  reduce system complexity, and improve the generalization performance \cite{jain2017non,liang2013feature}. We recall several motivating applications with \eqref{projection.lp} in the following.
\begin{itemize}
	\item[1.] \textit{The $\ell_{p}$-constrained sparse coding:} In the context of sparse coding \cite{Thom2015Efficient}, the task is to reconstruct the unknown sparse code word $\bm{x}^*\in\mathbb{R}^{n}$ from the linear measurements $\bm{y} = \bm{A}\bm{x}^* + \bm{\xi}$, where  $\bm{y}\in\mathbb{R}^{m}$ represents the data with $m$ features,  $\bm{\xi}\in\mathbb{R}^{m}$ denotes the noise vector, and  $\bm{A}$ corresponds to the fixed dictionary that consists of $n$ atoms with respect to its columns. This problem can be formulated as 
	\begin{equation}\label{eq: lp_regression}
		\min_{\bm{x} \in \mathbb{R}^n}~  \Vert\bm{A}\bm{x}-\bm{y}\Vert_{2}^{2}  \quad
		\textrm{s.t.}~  \Vert\bm{x}\Vert_{p}^{p} \leq \gamma,
	\end{equation}
	where the $\ell_{p}$ ball constraint is to induce sparsity in the code word. To solve  \eqref{eq: lp_regression},  the projected gradient descent (PGD) method is often employed \cite{bahmani2013unifying}. The efficiency of such an algorithm heavily relies  on the computation of projections onto the $\ell_p$ ball. 
	
	\item[2.] \textit{Projection-based transform domain denoising:} In standard wavelet denoising \cite{cetin2015projection, grandits2017optimizing}, it could be used as a prior information that  the unknown signal $\bm{x}$ is sparse \cite{chang2000adaptive}. Suppose $\bm{y}$ represents the wavelet signals   and the radius $\gamma$ is predetermined,  problem \eqref{projection.lp} is usually a pivotal procedure to be solved.
	
	\item[3.] \textit{Adversarial attacks for deep learning:}  Adversarial attacks have become a critical task for designing robust neural networks, which considers the security issues of deep neural networks \cite{yuan2019adversarial}. Among many advances in generating adversarial examples, a line of works considers the bounded norm-constraints of the perturbation \cite{su2019one}, which endeavors to operate a few entries merely. In particular, a central step is to compute the projection of the perturbation vector onto a $\ell_{p}$ ball in the black-box attacks setting \cite{balda2020adversarial}.
	
\end{itemize}

Despite its potential wide applicability,  the nonsmooth and nonconvex nature of \eqref{projection.lp} makes it challenging to solve in general. In principle, it is non-trivial to derive optimality conditions to characterize the optimal solutions of \eqref{projection.lp}. This introduces difficulties in numerical algorithm design and the analysis of the convergence properties of the proposed algorithms. 

Due to these aforementioned difficulties, to the extent of our knowledge, not much has been achieved for solving \eqref{projection.lp}. A pioneer study \cite{das2013non} proposed an exhaustive search algorithm, which combines branching and root-finding techniques. By focusing on the local optimality with respect to the nonzero components of the iterates for each branch, they obtained the optimal solution via finding the roots of a univariate nonlinear equation. It is claimed that the proposed algorithm could find a globally optimal solution for most $p$ with high probability. A recent study \cite{chen2019outlier} derived a unified characterization of the optimal solution and applied the Lagrangian method to improve  \cite{das2013non}, where a bisection method is used for finding the optimal multiplier; for each multiplier, this method needs to solve a system of nonlinear equations. By postulating that~\eqref{projection.lp} is addressed by a practical algorithm, the authors in \cite{bahmani2013unifying} theoretically investigated the performance of the PGD method in solving the $\ell_{p}$-constrained least squares problems. At this point, the authors of \cite{bahmani2013unifying} commented, \emph{``We believe developing an efficient algorithm for projection onto $\ell_{p}$ balls with $p\in(0,1)$ is an interesting problem that can provide a building block for other methods of sparse signal estimation involving the $\ell_{p}$-norm.   Furthermore, studying this problem may help to find an insight on how the complexity of these algorithms vary in terms of p.}"

The typical algorithmic frameworks for solving the nonconvex and nonsmooth optimization problems include the majorization-minimization (MM) framework \cite{lange2016mm} and the difference-of-convex (DC) framework \cite{le2018dc}. For the MM framework, a great variety of works considered the nonconvex $\ell_{p}$-penalized minimization problems and addressed it by solving a sequence of reweighted $\ell_{1}$-penalized subproblems \cite{figueiredo2007majorization,chartrand2008iteratively}. Our proposed algorithm also falls into the MM algorithmic framework since it generates the sequence by iteratively minimizing a sequence of weighted $\ell_{1}$-ball projection problems in which the constructed weighted $\ell_{1}$ ball is contained in the $\ell_{p}$ ball. It should be highlighted that the $\ell_{p}$ ball projection problem generalizes the $\ell_{p}$-penalized problem, as evidenced by \cite{das2013non}. On the other hand, the convergence results of the MM framework for constrained optimization investigated in \cite{bolte2016majorization} are not directly applicable to our algorithm due to the semi-algebraic constraint set assumption.  Another line of works investigated the DC programming framework for the sparse constrained optimization problems \cite{boob2020feasible}, but their surrogate functions for the $\ell_{p}$ ball constraint are different from ours. We would also like to mention that the updating strategy for the perturbation in our algorithm is novel and is crucial to the convergence analysis.

It is worthy to point out that fast convergence and low complexity properties are usually required when performing $\ell_{p}$ ball projections, since $\eqref{projection.lp}$ usually is a subroutine that is embedded in various algorithmic frameworks. The existing Lagrangian based methods mentioned above generally can not guarantee the feasibility of the obtained solution, and its effectiveness is only  demonstrated empirically. On the other hand, the performance of root-finding within Lagrangian based methods degrades in a large-scale setting. To overcome these challenges, we develop the first-order necessary optimality conditions to characterize the solutions of \eqref{projection.lp},  design numerical approaches for solving \eqref{projection.lp} and analyze the global convergence and convergence rate. The core idea of our work is based on the observation that the $\ell_p$ ball is locally equivalent to a weighted $\ell_1$ ball.  Therefore, our approach solves  \eqref{projection.lp} via a sequence of projections onto weighted $\ell_{1}$-balls, which are constructed by first adding perturbation to smooth the $\ell_p$ norm and then linearizing the $\ell_{p}$ norm at current iterate. Notably, the weighted $\ell_1$ ball projection subproblem can be solved accurately by finite terminating algorithms \cite{perez2020efficient} with $O(n)$ complexity, and these make our methods suitable for large-scale problems. Comparison with the existing state-of-the-art method highlights the advantage of the proposed algorithm which improves the  CPU time by at least $2$ orders of magnitude. Moreover, we develop a novel updating strategy for the perturbation parameters such that the algorithm is guaranteed to converge to a first-order stationary solution. The optimality errors are shown to converge with the rate of $O(1/\sqrt{k})$.

\subsection{Organization} In the remainder of this section, we outline our notation that will be employed throughout the paper.  First-order optimality conditions are derived in section \ref{sec2}. The proposed reweighted $\ell_{1}$-ball projection algorithm is detailed in section \ref{Sec: algorithm}. The global convergence and complexity are analyzed in section \ref{sec: analysis}. The results of numerical experiments are presented in section \ref{sec: exprements}. Concluding remarks are provided in section \ref{sec: Conclusion}.

\subsection{Notation and Preliminaries}
Throughout this paper, we limit our discussion to $\mathbb{R}^{n}$ with the standard inner product $\langle\bm{x},\bm{y}\rangle = \bm{x}^{T}\bm{y} = \sum_{i=1}^{n}x_iy_i$, and all boldface lower case letters represent vectors. In particular, $\bm{0}$ denotes the zero vector of appropriate size. Let $\mathbb{R}_{+}^{n}$ be the nonnegative orthant of $\mathbb{R}^{n}$, and $\mathbb{R}_{++}^{n}$ denote the interior of $\mathbb{R}_{+}^{n}$. Let $\mathbb{N} = \{1, 2, ... \}$ be the set of natural numbers.  For any $n\in \mathbb{N}$, we use $\left[ n\right]$ to represent the set of integers $\{1,2,\ldots,n\}$. 
We use $\|\cdot\|_1$ and  $\|\cdot\|_{2}$ to represent the $\ell_{1}$ norm and the $\ell_{2}$ norm of any vector $\bm{x}\in\mathbb{R}^{n}$, respectively. Define $|\bm{x}| = (|x_1|,\ldots,|x_{n}|)^{T}$, $\textrm{sign}(\bm{x}) = (\textrm{sign}(x_1), \ldots, \textrm{sign}(x_n))^T $ and $[\bm{x}]^{p} = (x_{1}^{p},\ldots,x_{n}^{p} )^{T}$ with  $p>0$. 
Given an index set $\mathcal{S} \subset \mathbb{N} $, we use $|\mathcal{S}|$ to denote the cardinality of $\mathcal{S}$.  Using this notation, we use $\bm{I}_{\mathcal{S}\mathcal{S}}$ to represent the submatrix of the identity matrix $\bm{I}$, whose rows and columns are indexed by $\mathcal{S}$. In addition, $\textrm{Diag}(\bm{x})\in\mathbb{R}^{n\times n}$ denotes a diagonal matrix with $\bm{x}$ forming the diagonal elements.   The Hadamard product $\bm{a}\circ \bm{b} \in \mathbb{R}^n$ for $\bm{a}, \bm{b}\in\mathbb{R}^n$ is defined element-wisely as $(\bm{a}\circ \bm{b})_i = a_i b_i$, $i\in[n]$. The active and inactive index set of $\bm{x}$ are respectively defined as 
$$ \mathcal{A}(\bm{x}) := \{i\mid x_i=0, i \in [n] \} \  \textrm{ and }\  \mathcal{I}(\bm{x}) := \{i\mid x_i \neq 0, i \in [n]\}.$$ 
Moreover, 
$x \sim N(\mu,\sigma^{2})$ denotes a Gaussian random variable with mean $\mu$ and variance $\sigma^{2}$; $z\sim U[a,b]$ corresponds to a random variable uniformly distributed on a closed interval $[a,b]$.

We next recall the concept of normal cone presented in \cite{rockafellar2009variational}, which plays a crucial role in establishing optimality conditions in nonlinear programming.
\begin{definition}
	Let $C \subset \mathbb{R}^{n}$ be closed and $\bar{\bm{x}} \in C$. A vector $\bm{v} \in \mathbb{R}^{n}$ is normal to $C$ at $\bar{\bm{x}}$ in the regular sense, written $\bm{v} \in \widehat{\mathcal{N}}_{C}(\bar{\bm{x}})$, if
	\begin{equation*}
		\limsup_{\substack{\bm{x}\underset{\mathcal{C}}{\rightarrow}\bar{\bm{x}}\\\bm{x}\neq\bar{\bm{x}}}} \frac{\la\bm{v}, \bm{x}-\bar{\bm{x}}\ra}{\Vert\bm{x}-\bar{\bm{x}}\Vert_{2}} \leq 0,\  \textrm{ for}\ \bm{x} \in C.
	\end{equation*}
	It is normal to $C$ at $\bar{\bm{x}}$ in the general sense, written $\bm{v} \in \mathcal{N}_{C}(\bar{\bm{x}})$, if there are sequences $\bm{x}^{\nu}\underset{\mathcal{C}}{\rightarrow} \bar{\bm{x}}$ and $\bm{v}^{\nu} \to \bm{v}$ with $\bm{v}^{\nu} \in \widehat{\mathcal{N}}_{C}(\bm{x}^{\nu})$.
\end{definition}

\section{Proposed Algorithm}\label{sec2} 

In this section, we characterize the properties of the optimal solutions of~\eqref{projection.lp} and derive the optimality conditions. We first show that the projection of $\bm{y}$  lies in the same orthant due to the symmetry of the $\ell_p$ ball. 
This property is proven to be true for $\ell_1$-ball projections \cite{condat2016fast, duchi2008efficient}. We show this property for a more general weighted $\ell_p$-ball projection problem, i.e.,
\begin{equation}\label{general.projection}
	\min_{\bm{x} \in \mathbb{R}^n}\ \frac{1}{2}\|\bm{x}-\bm{y}\|_2^2\quad 
	\st  \   \sum_{i=1}^n w_i |x_i|^p \leq \gamma 
\end{equation}
with $p\in(0,1]$ and $w_i \geq 0$, for all $i \in [n]$.  If $w_i = 1$ for all $i \in [n]$ and $p\in(0,1)$, this problem reverts to  \eqref{projection.lp}.  If $p=1$, we have a weighted $\ell_1$-ball projection problem. 
\begin{lemma} \label{projection.absolute}
	Let  $\bm{y}$ be a point satisfying  $\sum_{i=1}^n w_i |y_i|^p > \gamma$  and  $\bm{x}^*$ be a global optimal solution of \eqref{general.projection}.  Then, the following properties hold:
	\begin{itemize} 	
		\item[(i)] $x_i^*y_i  \ge 0$ and  $|x_i^*|  \leq |y_i|$  for all $i\in [n]$.
		
		\item[(ii)] $|\bm{x}^*|$ is a global optimal solution of 
		\begin{equation}\label{tmp.solution}
			\min\limits_{\bm{x}\in\mathbb{R}^n} \ \frac{1}{2}\|\bm{x}- |\bm{y}|\|_2^2 \quad \textrm{\textrm{\normalfont s.t.}}  \sum_{i=1}^n w_i x_i^p  \leq \gamma, \ x_i\in\mathbb{R}_+.
		\end{equation} 
	\end{itemize}
\end{lemma}
\begin{proof} 
	We prove (i) by contradiction. Assume there exists $j \in [n]$ for which $x_{j}^*  y_{j} < 0$.
	Now consider  $\tilde{\bm x}$ such that 
	\[\tilde{x}_j = 0\  \textrm{ and }\   \tilde{x}_i = x_i^*,  i\neq j.
	\]
	We can see that $\sum_{i=1}^n w_i |\tilde{x}_i|^p \leq \sum_{i=1}^n w_i |x_i^*|^p \leq \gamma$; hence $\tilde{x}$ is  feasible for \eqref{general.projection}. Note that $(x_j^* - y_j)^2 = (x_j^*)^2 - 2x_j^*y_j + y_j^2 > y_j^2 = (\tilde{x}_j-y_j)^2$ since $x_j^*y_j < 0$, 
	implying  $\|\bm{x}^*- \bm{y}\|_2^2 > \|\tilde{\bm{x}}- \bm{y}\|_2^2$.  This contradicts the global optimality of $\bm{x}^*$. Hence 
	$x_i^*y_i  \ge 0$  for all $i\in [n]$. 
	
	We also assume by contradiction that  there exists $j \in [n]$ for which $|x_j^*| > |y_j|$.
	Now consider  $\hat{\bm x}$ such that
	\[ \hat{x}_j = y_j \  \textrm{ and }\    \hat{x}_i=x_i^*, \forall i\in[n], i\neq j.\]
	We can see that $\sum_{i=1}^n w_i |\hat{x}_i|^p < \sum_{i=1}^n w_i  |x_i^*|^p = \gamma$; hence $\hat{\bm{x}}$ is feasible. Note that $(x_j^*-y_j)^2 > (\hat{x}_j-y_j)^2 = 0$,  which implies that $\|\bm{x}^*- \bm{y}\|_2^2 > \|\hat{\bm{x}}- \bm{y}\|_2^2$, contradicting the global optimality of $\bm{x}^*$.  Therefore, $|x_i^*|  \leq |y_i|$  for all $i\in [n]$. 
	
	As for (ii), notice that $\|\bm{x}^*- \bm{y}\|^2_2 \leq \|\bm{x}- \bm{y}\|^2_2$ for any feasible $\bm{x}$ for \eqref{tmp.solution}, which along with
	statement (i), implies that $\|\bm{x}^*-\bm{y}\|^2_2 = \|\bm{x}^*\|_2^2 - 2\langle\bm{x}^*,\bm{y}\rangle + \|\bm{y}\|_2^2 = \||\bm{x}^*|-|\bm{y}|\|^2_2   \leq \|\bm{x}-|\bm{y}|\|^2_2$ for any $\bm{x}$ feasible for \eqref{tmp.solution}.
	This proves that $|\bm{x}^*|$ is global optimal for \eqref{tmp.solution}.  
\end{proof}

Lemma \ref{projection.absolute} (i)  indicates  that we can restrict the discussions of $\ell_{p}$ ball projection problem on $\mathbb{R}_+^n$. In other words, instead of solving the original projection problem \eqref{general.projection}, we can instead compute the projection of $|\bm{y}|$ onto the $\ell_p$ ball in the positive orthant.  Letting $\bar{\bm{y}} := \vert\bm{y}\vert$, we can then focus on the following optimization problem
\begin{equation}\label{projection.lp.+}\tag{$\mathscr{P}$}
	\begin{aligned}
		\min\limits_{\bm{x}} & \quad f(\bm{x}):=  \frac{1}{2}\|\bm{x}-\bar{\bm{y}}\|_2^2 \\
		\st &\quad  \sum_{i=1}^n x_i^p \leq \gamma, \quad\   \bm{x} \in \mathbb{R}^{n}_{+}.
	\end{aligned}
\end{equation}
After obtaining the optimal solution $\bm{x}^{\star}$ to \eqref{projection.lp.+}, the solution of \eqref{projection.lp}  can thus be recovered by setting $  \textrm{sign}(\bm{y})\circ \bm{x}^{\star}$. In addition, we can see that the optimal solution satisfies $x_i^* = 0$ if $y_i=0$ by Lemma \ref{projection.absolute} (i).

\subsection{First-order Necessary Conditions for Optimality}
In this subsection, we derive the first-order necessary optimality conditions to characterize the solutions of \eqref{projection.lp.+} 
by referring to \cite{wang2021constrained}. For briefty, from now on, we use  $\Omega$ to denote the feasible region of~\eqref{projection.lp.+}, i.e., 
\begin{equation}\label{eq.omega}
	\Omega = \{ \bm{x} \mid  \sum_{i=1}^nx_i^{p} \leq \gamma, \  \bm{x}\in\mathbb{R}_+^n\}.
\end{equation}

By \cite[Theorem 6.12]{rockafellar2009variational}, any locally optimal solution of \eqref{projection.lp.+}  satisfies the following first-order necessary optimality condition. 
\begin{theorem}\label{thm.optimality}
	Let $\bm{x}$ be a local solution of \eqref{projection.lp.+}. Then the following holds
	\begin{equation}\label{eq:optimality}
		\bm{x} - \bar{\bm{y}} + \mathcal{N}_{\Omega}(\bm{x}) \ni \bm{0}.
	\end{equation}
\end{theorem}

To verify the satisfaction of the necessary condition \eqref{eq:optimality}, we need to further investigate the properties of $\mathcal{N}_{\Omega}(\bm{x})$. Note from our assumption that the point to be projected is not in the ball, and the local optimal solution $\bm{x}$ to~\eqref{projection.lp.+} lies on the boundary of the $\ell_{p}$ ball, i.e., $\sum_{i\in\Ical(\bm{x})} x_i^p = \gamma$. 
The following proposition characterizes the  elements of $\mathcal{N}_\Omega(\bm{x})$, and its proof mainly follows from \cite[Theorem 2.2 \& Theorem 2.6 (c)]{wang2021constrained} and \cite[Theorem 6.42]{rockafellar2009variational}.

\begin{proposition}\label{Pro.normal}
	Let $\bm{x} \in \Omega$. It holds that
	\begin{equation*}
		\mathcal{N}_{\Omega}(\bm{x}) = \left\lbrace 
		\begin{aligned}
			&\{\bm{v} \in \mathbb{R}^{n}_{+}\mid v_{i} = \lambda px_i^{p-1},i \in \Ical(\bm{x}); \lambda \geq 0\}, &~\textrm{ if } \sum_{i=1}^{n}x_i^{p} = \gamma,\\
			& \{\bm{0}\}, &~\textrm{ if } \sum_{i=1}^{n}x_i^{p} < \gamma.
		\end{aligned}\right.
	\end{equation*}
\end{proposition}

We depict the geometrical representation of the normal cone at $\bm{x}\in\Omega$ for a two-dimensional example in \Cref{fig:cone}.
\begin{figure}[H]
	\centering
	\includegraphics[width=3.5in]{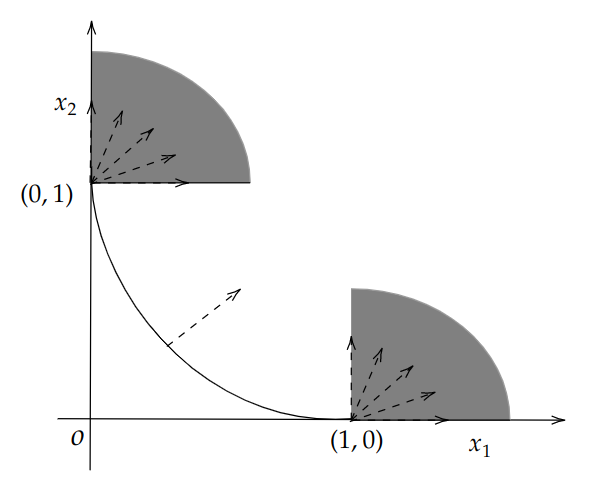}
	\caption{A two-dimensional illustration of the normal cone for the unit $\ell_{0.5}$-ball constraint in the nonnegative orthant. The normal cone has infinitely many normal directions at $(0,1)$ and $(1,0)$ as represented in the shaded region, and it has a single normal direction besides those two points. }
	\label{fig:cone}
\end{figure}

Using Proposition \ref{Pro.normal} and by~\Cref{thm.optimality}, the first-order necessary conditions for \eqref{projection.lp.+} are stated as 
\begin{subequations}\label{eq:general_kkt1}
	\begin{alignat}{2}
		(\bar{y}_i  - x_i)x_i   &= \ \lambda p x_{i}^{p}, \   i \in [n],\label{kkt.2}\\
		\sum_{i\in[n]} x_i^p  &= \  \gamma, \label{kkt.3}\\
		\bx \ge 0, \ 	\lambda  &\ge\ 0.\label{kkt.5} 
	\end{alignat}
\end{subequations}
where $\lambda$ is the dual variable/multiplier.

\section{An Iteratively Reweighted $\ell_1$ Ball Projection Algorithm}\label{Sec: algorithm}
In this section, we describe our proposed algorithm for solving \eqref{projection.lp.+}, which can be viewed as an instance of the MM framework. We use a weighted $\ell_1$ ball as a majorizer of the perturbed $\ell_{p}$ ball in the majorization step, and the majorizer is constructed by relaxing and linearizing the $\ell_p$ ball constraint. In the minimization step, the next iterate is computed as the projection of $\bar{\bm{y}}$ onto the weighted $\ell_1$ ball. In this way, our algorithm needs to solve a sequence of weighted $\ell_1$ ball projection subproblems, which can be solved efficiently within a finite number of iterations.  

\subsection{Weighted $\ell_{1}$-ball Approximation} 
Now we describe our weighted $\ell_1$ ball projection subproblem. First, 
we add a perturbation $\epsilon_i > 0$ to have a continuous differentiable approximate model with respect to $|x_i|$, 
\begin{equation*}
	\vert x_i\vert^{p} \approx (\vert x_i\vert+\epsilon_i)^{p}.
\end{equation*} 
Correspondingly, we denote the perturbed $\ell_{p}$ ball constraint in~\eqref{projection.lp.+} by $\Omega_{\bm{\epsilon}} := \{\bm{x}\mid \sum_{i=1}^{n} (x_{i} + \epsilon_{i})^{p}\leq\gamma, \bm{x}\in\mathbb{R}^{n}_{+}\}$. Then at $\bm{x}^k$ with perturbation $\bm{\epsilon}^{k} > \bm{0}$, by concavity of the $\ell_{p}$ norm, we approximate $(\vert x_i\vert+\epsilon_i^k)^p$ by linearizing it with respect to $\vert x_i^k\vert$ to have 
\begin{equation*}
	(\vert x_i\vert+\epsilon_i^k)^p  \leq  (|x_i^k|+\epsilon_i^k)^p +  p( |x_i^k| + \epsilon_i^k) ^{p-1} (|x_i| -|x_i^k|).
\end{equation*}
Therefore, at the $k$th iterate $\bm{x}^k$, the constraint in~\eqref{projection.lp.+} can be replaced by 
\begin{equation}\label{linearize}
	\sum_{i=1}^n (|x^k_i| +\epsilon_i^k)^p + \sum_{i=1}^n  w_i^k (|x_i|-|x_i^k|) \leq \gamma
\end{equation} 
with $ w_i^k := w(x_i^k, \epsilon_i^k) = p(|x_i^k| + \epsilon_i^k)^{p-1}$.

Defining $\gamma^k :=  \gamma -  \sum_{i=1}^n (x_i^k+\epsilon_i^k)^p +  \sum_{i=1}^n  w_i^k |x_i^k|$ and 
rearranging,  we have the $k$th convex relaxation of \eqref{projection.lp.+}, which reads
\begin{equation}\label{subproblem1}\tag{$\mathscr{P}_{\textrm{sub}}$}
	\begin{aligned}
		\min\limits_{\bm{x}} & \quad  \frac{1}{2}\|\bm{x}-\bar{\bm{y}}\|_2^2\\
		\st & \quad\bm{x} \in \Omega_{\bm{w}^k}:= \{\bm{x}\mid\sum_{i=1}^n w_i^k x_i \leq \gamma^k,\ \bm{x}\in\mathbb{R}^n_+\}.
	\end{aligned}
\end{equation}

The feasible region of \eqref{subproblem1} is in fact a weighted $\ell_1$ ball  intersected with the nonnegative orthant. Since $\bar{\bm{y}}\in\mathbb{R}_+^n$,   \eqref{subproblem1}  is clearly  equivalent to the projection of $\bar{\bm{y}}$  onto the weighted $\ell_1$ ball $\Omega_{\bm{w}^k}$. 

The next iterate $\bm{x}^{k+1}$ is then computed as the optimal solution of \eqref{subproblem1}.  
Therefore, $\bm{x}^{k+1}$ satisfies the following  Karush-Kuhn-Tucker  (KKT) conditions for \eqref{subproblem1}
\begin{subequations}\label{eq: subproblem_kkt}
	\begin{alignat}{2}
		\bar{y}_i -  x_i^{k+1} - \lambda^{k+1} w_i^{k} + \mu_i^{k+1} &=0,  \ \  i\in [n] \label{kkt.tmp.1}\\
		\langle \bm{\mu}^{k+1}, \bm{x}^{k+1} \rangle & =  0,    \label{kkt.tmp.2}\\
		\mu_i^{k+1} &\ge  0,  \ \ i\in[n] \label{kkt.tmp.3}\\
		\sum_{i\in\Ical(\bm{x}^{k+1})} w_i^{k} x_i^{k+1} &\le   \gamma^{k}, \label{eq:ball-constraint}\\
		\lambda^{k+1} &\ge 0,\\
		\lambda^{k+1}(\sum_{i\in\Ical(\bm{x}^{k+1})} w_i^{k} x_i^{k+1} -\gamma^{k}) &= 0,\label{kkt.3.l1}
	\end{alignat}
\end{subequations} 
where $ \lambda^{k+1}$ and $\bm{\mu}^{k+1} $ are multipliers. 
Conditions \eqref{kkt.tmp.1}--\eqref{kkt.tmp.3} can be equivalently written as 
\begin{subequations}\label{eq: subproblem_kkt_real}
	\begin{alignat}{2}
		\bar{y}_i - \lambda^{k+1} w_i^{k} &\le 0,  \  i\in \Acal(\bm{x}^{k+1})\label{kkt.2.l1},\\
		\bar{y}_i - x_i^{k+1} -\lambda^{k+1} w_i^{k}&= 0,   \ i\in\Ical(\bm{x}^{k+1}) \label{kkt.1.l1}.
	\end{alignat}
\end{subequations} 
Clearly,  $\lambda^{k+1}$ satisfies 
\begin{equation}\label{lambda.l1}
	\lambda^{k+1}  = \frac{ \sum\limits_{i\in\Ical(\bm{x}^{k+1})} ( \bar{y}_i - x_i^{k+1} ) }{\sum\limits_{i\in\Ical(\bm{x}^{k+1})} w_i^k}. 
\end{equation}

\subsection{A Dynamic Updating Strategy for \boldmath{$\epsilon$}} 

Our algorithm solves a sequence of weighted $\ell_1$ ball projection subproblems with the relaxation parameters $\epsilon_i, i\in [n]$ driven to $0$. However, if we decrease $\epsilon_i, i\in [n]$ too slowly, then the algorithm may suffer from unsatisfactorily slow convergence. On the other hand, if $\epsilon_i, i\in [n]$ were driven to 0 too fast,  then the algorithm may quickly get trapped in local solutions or even fail to converge to stationary points. Unlike the usual parameter tuning strategy in the penalized model which tries a sequence of penalty parameters to approximately find one satisfying the needs, we design a dynamic updating strategy \cite{burke2015iterative} that can automatically determine when $\epsilon_i, i\in[n]$ should be reduced. We update $\epsilon_i, i\in[n]$  according to the following rule
\begin{equation}\label{update.condition} 
	\begin{cases} 
		\bm{\epsilon}^{k+1} \in (\bm{0}, \theta \bm{\epsilon}^k],   & \text{ if } \|\bm{x}^{k+1} - \bm{x}^k\|_2  \| \text{sign}(\bx^{k+1}-\bx^k)\circ  \bw^{k}\|_{2}^{\tau } \le M, \\
		\bm{\epsilon}^{k+1} = \bm{\epsilon}^k,  & \text{ otherwise}. 
	\end{cases}
\end{equation} 
with $0<\theta<1$, $M > 0$ and $\tau>1$. As discussed in the numerical study, we can determine these parameters with a rough tuning, in particular $\tau$ could be selected close to $1$.

Our proposed  Iteratively Reweighted $\ell_1$-Ball Projection (IRBP) algorithm for solving  \eqref{projection.lp.+} 
is stated in  \Cref{alg.l1}.
\begin{algorithm}[H]
	\caption{Iteratively Reweighted $\ell_1$ Ball Projection Algorithm (IRBP)}
	\label{alg.l1}
	\begin{algorithmic}[1]
		\STATE Input:  $\bm{y}$, $p$, and $\gamma$. Let $\bar{\bm{y}} = |\bm{y}|$.  Choose  $\tau > 1$, $0<\theta<1$  and $M > 0$. 
		\STATE (Initialization)  Choose $(\bm{x}^0, \bm{\epsilon}^0) \in \mathbb{R}^n_+\times \mathbb{R}^n_{++}$ such that $\Vert\bm{x}^{0}\Vert_{p}^{p}\leq \gamma$, $\Vert\bm{\epsilon}^{0}\Vert_{p}^{p}\leq\gamma$ and $\| \bm{x}^0 + \bm{\epsilon}^0\|_p^p \le \gamma$.  Set  $k = 0$.
		\STATE (Compute the weights) Set   $w_i^k = p(|x_i^k| + \epsilon_i^k)^{p-1}$ for $i\in\left[ n \right] $.
		\STATE (Solve the subproblem for $\bm{x}^{k+1}$) Compute $\bm{x}^{k+1}$ as the optimal solution of \eqref{subproblem1}.		
		\STATE (Set the new relaxation vector $\bm{\epsilon}^{k+1}$) Obtain $ \bm{\epsilon}^{k+1}$ according to
		\begin{equation*}
			\begin{cases} 
				\bm{\epsilon}^{k+1} \in (\bm{0}, \theta \bm{\epsilon}^k],   & \text{ if } \|\bm{x}^{k+1} - \bm{x}^k\|_2  \| \text{sign}(\bx^{k+1}-\bx^k)\circ  \bw^{k}\|_{2}^{\tau } \le M, \\
				\bm{\epsilon}^{k+1} = \bm{\epsilon}^k,  & \text{ otherwise}. 
			\end{cases}
		\end{equation*} 
		\STATE Set $k\gets k+1$ and go to step~$3$. 
	\end{algorithmic}
\end{algorithm}

The initialization requirement of $(\bm{x}^0, \bm{\epsilon}^0)$ in~\Cref{alg.l1} is easy to fulfill.  For example, we can simply choose $ \epsilon_i^0 = (\frac{\gamma}{n})^{1/p}$ such that $\|\bm{\epsilon}^0\|_p^p < \gamma$ and then set $\bm{x}^0=\bm{0}$. On the other hand, the weighted $\ell_1$-ball projection subproblem of the form \eqref{subproblem1} can be addressed by calling the exact projection algorithms proposed in \cite{michelot1986finite,condat2016fast, zhang2020inexact} with a straightforward modification. 

\section{Convergence Analysis of IRBP}\label{sec: analysis}
This section is devoted to the convergence analysis of the proposed~\Cref{alg.l1}, including the well-posedness, global convergence, conditions guaranteeing unique convergence, and the worst-case complexity of the optimality residuals.

\subsection{Well-posedness and Basic Properties} 

The following theorem guarantees that the subproblems of our proposed algorithm are always feasible so that the algorithm is well defined. 

\begin{theorem}[Well-posedness]\label{lemma.wellpose}
	It holds true that $\|\bm{x}^k + \bm{\epsilon}^k\|^p_p \le \gamma, \ \forall k \in \mathbb{N}$. Hence $\gamma^k>0$, and all subproblems in~\Cref{alg.l1} are feasible.
\end{theorem} 

\begin{proof} 
	First, $\|\bm{x}^k + \bm{\epsilon}^k\|^p_p \le \gamma$ implies that the $k$th subproblem is feasible, since in this case $\bm{x}^k$ satisfies  the constraint of the $k$th subproblem 
	\[ \|\bm{x}^k+\bm{\epsilon}^k\|_p^p + \sum_{i=1}^n p(x_i^k+\epsilon^k_i)^{p-1}(x_i -x_i^k)
	\le \gamma. \]
	Therefore, we only have to prove $\|\bm{x}^k + \bm{\epsilon}^k\|^p_p \le \gamma$ holds true for all $k$.

	Since the initial point is chosen such that $\|\bm{x}^0 + \bm{\epsilon}^0\|^p_p \le \gamma$, by induction, 
	we only have to show that $\|\bm{x}^{k+1} + \bm{\epsilon}^{k+1}\|^p_p \le \gamma$ provided that
	$\|\bm{x}^k + \bm{\epsilon}^k\|^p_p \le \gamma$. 
	$\bm{x}^{k+1}$  satisfies the constraint of the $k$th subproblem, since $\bm{x}^{k+1}$ is the optimal solution to the $k$th subproblem. By concavity of $(\cdot)^{p}$, we have  
	\begin{equation*}
		\begin{aligned}
			\|\bm{x}^{k+1}+\bm{\epsilon}^k\|_p^p \le & \  \|\bm{x}^k+\bm{\epsilon}^k\|_p^p + \sum_{i=1}^n p(x_i^k+\epsilon^k_i)^{p-1}(x_i^{k+1}-x_i^k)
			\le \gamma, 
		\end{aligned}
	\end{equation*}
	which, together with   $\epsilon_i^{k+1} \le \epsilon_i^k$ for all $i\in [n]$ and $\bm{x}^{k+1}\in\mathbb{R}^n_+$, implies that  $\|\bm{x}^{k+1}+\bm{\epsilon}^{k+1}\|_p^p \le \|\bm{x}^{k+1}+\bm{\epsilon}^k\|_p^p \le \gamma$. This completes the proof.
\end{proof}

In the following lemma, we prove that the objective of~\eqref{projection.lp.+} has a quadratic improvement after each iteration, and the displacement of the iterates vanishes in the limit.

\begin{lemma}[Descent property]\label{lem.x-x}
	Let the sequence $\{\bm{x}^{k}\}$ be generated by~\Cref{alg.l1}. It holds that 
	\begin{itemize}
		\item[(i)] The sequence $\{\|\bm{x}^k-\bar{\bm{y}}\|_2^2\}$  monotonically decreases. Indeed,
		\begin{equation}\label{eq: sum1}
			\|\bm{x}^k-\bar{\bm{y}}\|_2^2 - \|\bm{x}^{k+1}-\bar{\bm{y}}\|_2^2   \geq \|\bm{x}^{k}-\bm{x}^{k+1}\|_2^2, \ \forall k \in \mathbb{N}.
		\end{equation}
		
		\item[(ii)] The sequence $\{\bm{x}^{k}\}$ satisfies $ \sum\limits_{k=0}^{+\infty} \| \bm{x}^{k+1} - \bm{x}^k\|_2^2 < +\infty$, and consequently, $ \lim\limits_{k\to +\infty} \| \bm{x}^{k+1} - \bm{x}^k\|_2^2 = 0$.
	\end{itemize} 
\end{lemma}
\begin{proof}
	Since $\bm{x}^{k}$ is feasible to \eqref{subproblem1}, the first-order optimality condition for projection onto a closed convex set yields the desired result of statement (i). For example, see \cite[Theorem 2.1 (2)]{burke2015iterative} and \cite{zarantonello1971projections} for the proof details.
	
	As for statement (ii), summing up both sides of~\eqref{eq: sum1} from $k=0$ to $t$ and rearranging, we have 
	\begin{equation*}
		\begin{aligned}
			\sum_{k=0}^t \|\bm{x}^{k+1}- \bm{x}^k\|_2^2 \leq   \sum_{k=0}^t( \|\bm{x}^k-\bar{\bm{y}}\|_2^2 - \|\bm{x}^{k+1} -\bar{\bm{y}}\|_2^2 ) &= \|\bm{x}^0 - \bar{\bm{y}}\|_{2}^{2} - \|\bm{x}^{t+1} - \bar{\bm{y}}\|_2^2\\ &\le \|\bm{x}^0 - \bar{\bm{y}}\|_{2}^{2}< +\infty.
		\end{aligned}
	\end{equation*}
	Letting $t\to\infty$, we conclude that $\|\bm{x}^{k+1}-\bm{x}^k\|_2^2\to 0$ as $k\to +\infty$, completing the proof.
\end{proof}

The next lemma enumerates relevant properties of the primal-dual solution to the subproblem \eqref{subproblem1}.

\begin{lemma}\label{lemma.primal}
	The following statements hold true:
	\begin{itemize}
		\item[(i)]  $\Ical(\bm{x}^k)\ne\emptyset$ and $0 \le x^k_i \le \bar{y}_i$ for any $i\in[n]$ and $k \in \mathbb{N}$. 
		
		\item[(ii)]  $\{\bm{x}^k\}_{k\ge 1}$ is uniformly bounded away from $\bm{0}$. 
		
		\item[(iii)]  $\sum_{i=1}^nw_i^k x_i^{k+1} = \gamma^k$.
		
		\item[(iv)]   $\{\lambda^k\}$ is bounded above.
	\end{itemize}
\end{lemma}
\begin{proof}
	
	\begin{itemize}
		\item[(i)]  For any $k$, $w_i^k < +\infty$ since $\epsilon_i^k > 0$. Therefore, the interior of the   
		feasible set of the projection subproblem \eqref{subproblem1} is not empty. Lemma \ref{projection.absolute} shows that $\bm{0}$ is not the projection 
		of $\bar{\bm{y}}$ onto the weighted $\ell_1$ ball, meaning $\Ical(\bm{x}^{k+1}) \neq \emptyset$. 
		The rest of statement (i) is simply a special case of Lemma \ref{projection.absolute} with $p=1$. 
		
		\item[(ii)] 
		We first show that $\{\bm{x}^k\}_{k\ge \bar k}$ is uniformly bounded away from $\bm{0}$, where $\bar k$ is such that $\bm{x}^{\bar k} \neq \bm{0}$. 
		From Lemma \ref{lem.x-x} (i), we have $\|\bm{x}^k - \bar{\bm{y}} \|_2 \le  \| \bm{x}^{\bar k} - \bar{\bm{y}}\|_2$ for any $k \ge \bar k$. 
		Therefore,  for any $k \ge \bar k$,  $ \|\bm{x}^k\|_2 \ge \|\hat{\bm{x}}\|_2$ with
		\begin{equation}\label{eq.tmp.xhat}
			\hat{\bm{x}} : =\arg \min_{\bm{x}}\{ \|\bm{x}\|_2^2 \mid \| \bm{x} - \bar{\bm{y}}\|_2^2 \le \| \bm{x}^{\bar k} - \bar{\bm{y}}\|_2^2   \}.
		\end{equation}
		It can be easily shown that $ 0< \|\bar{\bm{y}}\|_2 -  \| \bm{x}^{\bar k} - \bar{\bm{y}}\|_2 \leq \|\hat{\bm{x}}\|_2$~(please refer to~\Cref{app_A} for its proof), where the inequality is by $\bm{x}^{\bar k} \ne \bm{0}$ and statement (i). Therefore, $\{\|\bm{x}^k\|_2\}_{k\ge 1}$ is uniformly bounded away from 0.
		
		Now it suffices to show that $\bm{x}^{1} \neq \bm{0}$. If $\bm{x}^0 \neq \bm{0}$, then it is done. 
		If   $\bm{x}^0 = \bm{0}$, then  $ \|\bm{\epsilon}^0 \|_p^p < \gamma$ and the initial subproblem is 
		\begin{equation}\label{eq: (iii)}
			\min_{\bm{x}} \ \frac{1}{2}\|\bm{x}-\bar{\bm{y}}\|_2^2,\quad \st \  \sum_{i=1}^n p(\epsilon_i^0)^{p-1} x_i \le \gamma - \|\bm{\epsilon}^0 \|_p^p, \ \bm{x}\in\mathbb{R}_+^n.
		\end{equation} 
		Consider 
		$\hat{\bm{x}}: = \frac{\gamma - \|\bm{\epsilon}^0\|_p^p}{\sum\limits_{i=1}^n p(\epsilon_i^0)^{p-1} \bar{y}_i} \bar{\bm{y}}$.  It is clear that 
		$\hat{\bm{x}}$ is feasible for \eqref{eq: (iii)} and $\tfrac{1}{2}\|\hat{\bm{x}} - \bar{\bm{y}}\|_2^2 < \tfrac{1}{2}\|\bm{0} - \bar{\bm{y}}\|^2_2$, implying
		$\bm{x}=\bm{0}$ is not optimal for \eqref{eq: (iii)}. Therefore, $\bm{x}^{1}\neq \bm{0}$, completing the proof of statement~(ii). 
		
		\item[(iii)] Suppose by contradiction that 
		the projection of $\bar{\bm{y}}$  onto the weighted $\ell_1$ ball of \eqref{subproblem1} 
		satisfies $\sum_{i=1}^n w_i^k x_i^{k+1} < \gamma^k$,   it happens that 
		$\bar{\bm{y}}$ is also in the weighted $\ell_1$ ball, i.e., $\sum_{i=1}^n w_i^k \bar y_i \le \gamma^k$. Therefore, 
		\begin{equation}
			\begin{aligned}
				\|\bar{\bm{y}}  \|_p^p \!\le\! \|\bar{\bm{y}} +\bm{\epsilon}^k\|_p^p \!\le\! & \ \|\bm{x}^k+\bm{\epsilon}^k\|_p^p + \sum_{i=1}^n p(x_i^k+\epsilon^k_i)^{p-1}(\bar y_i-x_i^k)
				\le \gamma, 
			\end{aligned}
		\end{equation}
		contradicting our assumption that $\|\bar{\bm{y}} \|_p^p > \gamma$. This completes the proof of statement (iii).
		
		\item[(iv)] For any $i\in\Ical(\bm{x}^k)$, 
		$\{w_i^k\}$ is   bounded away from $0$  and 
		$\sum_{i\in\Ical(\bm{x}^k)} (\bar{y}_i - x_i^k)$ is bounded above according to statement (i).
		It follows from  \eqref{lambda.l1} that $\{\lambda^k\}$ is bounded above.
	\end{itemize}
\end{proof}

\subsection{Convergence Results}\label{sec.global}
In this section, we prove that any cluster point of the sequence generated by the IRBP is first-order stationary for~\eqref{projection.lp}. To enforce the convergence of the iterates to the stationary points of the original problem, our goal is to drive $\bm{\epsilon}^k$ to $\bm{0}$ as the algorithm proceeds, so that the relaxed ball $\|\bm{x}+\bm{\epsilon}\|_p\le \gamma$ can approximate the $\ell_p$ ball in the limit. We first show that the update \eqref{update.condition} of $\bm{\epsilon}$  is triggered infinitely many times. Define the set of iterations where $\bm{\epsilon}$ is decreased as
\begin{equation*}
	\Ucal:= \{ k \in \Nb \mid  \|\bm{x}^{k+1} - \bm{x}^k\|_2   \|\text{sign}(\bx^{k+1} - \bx^k) \circ  \bw^{k}\|_{2}^{  \tau }  \le M \}.
\end{equation*}
Then we have the following result, and its proof is provided in~\Cref{app_Trigger-Inf}.
\begin{lemma}\label{lem.eps to 0}
	The updating condition~\eqref{update.condition} is triggered infinitely many times, i.e., $|\Ucal| = +\infty$, \text{indicating} $\lim\limits_{k\to +\infty} \bm{\epsilon}^{k} = \bm{0}$.
\end{lemma}

For ease of presentation, we define four cluster point sets as follows.
\begin{align*}
	\Gamma &:=  \{ \text{cluster points of } \{(\bm{x}^{k}, \lambda^k)\} \},\quad
	\Gamma_\Ucal :=  \{ \text{cluster points of } \{(\bm{x}^{k}, \lambda^k)\}_\Ucal \}, \\
	\Lambda &:= \{ \text{cluster points of } \{\lambda^k\} \}, \qquad\quad\
	\Lambda_\Ucal := \{ \text{cluster points of } \{\lambda^k\}_\Ucal \}.
\end{align*}
The following property shows that for sufficiently large $k$, the zero and nonzero components of $\{\bm{x}^k\}_{\Ucal}$ remain unchanged. For simplicity, we use shorthands  $\Ical^k$ and $\Acal^k$ to represent $\Ical(\bx^k)$ and $\Acal(\bx^k)$, respectively.  Please refer to~\Cref{app_3importantProperty} for its proof.
\begin{proposition}\label{lem.stable_nonzero}
	Assume  $0 \not\in \Lambda_\Ucal$, then the following statements hold:
	\begin{itemize}
		\item[(i)] There exists $\bar{k}$ such that $\Ucal \supseteq \{k\mid k \geq \bar{k}\}$, which implies the update of $\beps$ in \eqref{update.condition} is always triggered for sufficiently large $k$.
		
		\item[(ii)] If $\{w_i^k\}$ is unbounded, then $x_i^k \equiv 0$ for all sufficiently large $k$.  
		If $\{w_i^k\}$ is bounded above, then there exists $\delta_{i} > 0$ such that $x_i^k > \delta_i > 0$ for all sufficiently large $k$. 
		
		\item[(iii)] There exist index sets $\Ical^*\cup \Acal^* = [n]$ such that 
		$\Ical(\bm{x}^*) = \Ical^*$ and $\Acal(\bm{x}^*) = \Acal^*$ for any $(\bm{x}^*,\lambda^{*})\in \Gamma$. 
		Furthermore,  $\Ical(\bm{x}^{k}) = \Ical^*$ and $\Acal(\bm{x}^k) = \Acal^*$  for all sufficiently large $k$.	\end{itemize}
\end{proposition}

The next corollary follows directly from Proposition \ref{lem.stable_nonzero}.
\begin{corollary}\label{Lem: coincidence}
	It holds that 
	\begin{itemize}
		
		\item[(i)] $\Lambda_\Ucal \subseteq \Lambda.$
		
		\item[(ii)] If $0 \not\in \Lambda$, then $ \Lambda = \Lambda_\Ucal.$
		
		\item[(iii)] If $0 \in \Lambda$, then $0 \in \Lambda_\Ucal.$
	\end{itemize}
\end{corollary}

Considering a primal-dual pair $(\bm{x}, \lambda)$ with $\bx\ge0$ and $\lambda > 0$, 
we define the following metrics to measure the optimality residuals at  $(\bm{x}, \lambda)$
\begin{equation}\label{eq: termination_rule}
	\begin{aligned}
		\alpha(\bm{x},\lambda)  :=  \sum_{i=1}^n \lvert (\bar{y}_i -x_i ) x_i -  \lambda p x_i^p \rvert,\quad \beta(\bm{x}) :=\left|\sum_{i =1}^n x_i^p  - \gamma\right|.
	\end{aligned}
\end{equation}
It is obvious that $(\bm{x}, \lambda)$  is first-order stationary if and only if  $\alpha(\bm{x}, \lambda)=0$ and $\beta(\bm{x})=0$.

The following property of the optimality residuals is needed in our analysis, and its proof is provided in~\Cref{app_C}.

\begin{lemma}\label{lem.residual bounded}
	There exists 
	$\tilde{\bm{x}}$ with $\tilde{x}_i^k$   between $x_i^k $ and $x_i^{k-1}+\epsilon_i^{k-1}$, $i\in\Ical^k$  such that 
	\begin{equation}\label{residual.bound1} 
		\alpha(\bm{x}^{k},\lambda^{k}) \leq \  \lambda^{k}p (1-p) \Vert[\tilde{\bm{x}}^{k}_{\Ical^k}]^{p-2} \Vert_{2}\Vert\bar{\bm{y}}\Vert_{\infty}\left[ \Vert\bm{x}^{k-1}_{\Ical^k} - \bm{x}^{k}_{\Ical^k}\Vert_{2} +\Vert\bm{\epsilon}^k_{\Ical^k}\Vert_{2}\right].
	\end{equation}
	There exists 
	$\hat{\bm{x}}$ with $\hat{x}_i^k$ between $x_i^k $ and $x_i^{k-1}+\epsilon_i^{k-1}$, $i\in \Ical^k$  such that 
	\begin{equation}\label{residual.bound2}
		\begin{split}
			\beta(\bm{x}^{k})  \leq &\  (\|\bw_{\Ical^{k}}^{k-1} \|_2+p\|[\hat\bx_{\Ical^{k}}^k]^{p-1} \|_2) \Vert\bm{x}^{k-1}_{\Ical^{k}} - \bm{x}^{k}_{\Ical^{k}}\Vert_{2} \\
			&\ + p\|[\hat\bx_{\Ical^{k}}^k]^{p-1} \|_\infty \|\bm{\epsilon}_{\Ical^{k}}^{k-1} \|_1
			+\| [ \beps_{\Acal^{k}}^{k-1} ]^p \|_1 + \| [ \beps_{  \Acal^{k}}^{k-1} ]^{p-1} \circ  {\bm{x}}_{\Acal^{k}}^{k-1}\|_1.
		\end{split}
	\end{equation}
\end{lemma}

The main result on the global convergence is stated in the following theorem.
\begin{theorem}\label{thm.opt.u}
	Let $\{(\bm{x}^k, \lambda^k)\}$ be the sequence generated by \Cref{alg.l1}. Then for all $ (\bm{x}^{*}, \lambda^*) \in \Gamma_\Ucal$, $(\bm{x}^*, \lambda^*)$ is first-order stationary for \eqref{projection.lp.+}. Consequently, we have 
	\begin{equation}
		\begin{aligned}
			\lim\limits_{\substack{k\to +\infty \\ k\in \Ucal}} \alpha(\bx^k, \lambda^k)  = 0 \ \textrm{ and } \
			\lim\limits_{\substack{k\to +\infty \\ k\in \Ucal}} \beta(\bx^k) = 0.
		\end{aligned}
	\end{equation}
\end{theorem}

\begin{proof}
	Given $(\bx^*, \lambda^*) \in \Gamma_\Ucal$, there exists
	$\Scal\subseteq \Ucal$ such that $\lim\limits_{ \substack{k\to+\infty\\k\in\Scal}} (\bm{x}^k, \lambda^k) = (\bm{x}^*,\lambda^*)$. It is easily seen that there exists $\bar{k} > 0 $ such that $\Ical(\bx^*) \subseteq \Ical(\bx^k)$ for all $k \ge \bar{k} \text{ and } k \in \Scal$. Lemma \ref{lemma.primal} (i) implies that $\Ical(\bx^*)$ is non-empty.
	\begin{equation}
		\begin{aligned}
			\alpha ( \bx^*, \lambda^*) & = \lim\limits_{ \substack{k\to+\infty\\k\in\Scal}} \alpha (\bx^k, \lambda^k) \\
			&  \leq \ \lim\limits_{ \substack{k\to+\infty\\k\in\Scal}} \lambda^{k}p (1-p) \Vert[\tilde{\bm{x}}^{k}_{\Ical^k}]^{p-2} \Vert_{2}\Vert\bar{\bm{y}}\Vert_{\infty}\left[ \Vert\bm{x}^{k-1}_{\Ical^k} - \bm{x}^{k}_{\Ical^k}\Vert_{2} +\Vert\bm{\epsilon}^k_{\Ical^k}\Vert_{2}\right] \\
			& = 0,
		\end{aligned}
	\end{equation}
	where the inequality is by Lemma \ref{lem.residual bounded} and the second equality is by Lemma \ref{lem.x-x} (ii) and Lemma \ref{lem.eps to 0}.
	
	Applying the same argument for $\beta(\bx^k)$, we have 
	$$
	\beta (\bx^*) = \lim\limits_{ \substack{k\to+\infty\\k\in\Scal}} \beta (\bx^k) = 0.$$ 
	Hence $(\bx^*, \lambda^*)$ is stationary for \eqref{projection.lp.+}.
	
	Overall, we have shown for any convergent subsequence $\Scal \subseteq \Ucal$,
	$$
	\lim\limits_{\substack{k\to +\infty \\ k\in \Scal}} \alpha(\bx^k, \lambda^k)  = 0 \ \textrm{ and }\ 
	\lim\limits_{\substack{k\to +\infty \\ k\in \Scal}} \beta(\bx^k) = 0,
	$$
	hence, both $\alpha(\bx^k, \lambda^k)$ and $\beta(\bx^k)$ converge to 0 on $\Ucal$.
\end{proof}

\subsection{Uniqueness of the Cluster Points} \Cref{thm.opt.u} shows the convergence of optimality residuals and all cluster points of $\{\bx^k\}_\Ucal$ are first-order stationary for \eqref{projection.lp.+}. In this subsection, we discuss conditions to guarantee the convergence of the entire sequence $\{\bx^k\}$.

The following lemma states the properties of the cluster set $\Gamma$.
\begin{lemma}\label{prop.property of limits}
	$\Gamma$ is non-empty, compact and connected. For every cluster point $(\bm{x}^*, \lambda^*) \in \Gamma$, they have the same objective value.
\end{lemma}

\begin{proof}  
	$\{(\bm{x}^k, \lambda^k)\}$ is bounded by Lemma \ref{lemma.primal} (i) and (iv), hence it has  cluster points. $\Gamma$ is compact and connected by Lemma \ref{lem.x-x} (ii) and \cite[Lemma 2.6]{bauschke2015} \cite[Theorem 26.1]{ostrowski1973solutions}. Lemma \ref{lem.x-x} (i) shows all points in $\Gamma$ have the same objective value. This completes the proof. 
\end{proof}

The following theorem asserts the convergence property of the entire sequence generated by IRBP, and its proof can be found in~\Cref{app_unique_convergence}.
\begin{theorem}\label{Theo: unique_convergence}
	Assume $0 \not\in \Lambda$ and there exists a cluster point $\bm{x}^{*}$ of $\{\bm{x}^{k}\}$ such that $x_i^* \neq [\lambda^* p(1-p)]^{\frac{1}{2-p}},  \forall i\in \Ical^*.$ Then $\lim\limits_{k\to +\infty}\bm{x}^k = \bm{x}^*$.
\end{theorem}

We have shown that when $0 \notin \Lambda$, $\{\bm{x}^k\}$ converges to a unique point. Next we demonstrate that when $0 \in \Lambda$, to check if a point $(\bm{x}^*, 0)$ is stationary for 
\eqref{projection.lp.+} or not is NP-hard. 

\begin{proposition}[NP-hardness]\label{prop.np} 
	It is NP-hard to determine whether \eqref{projection.lp.+}  has a 
	first-order stationary point  $(\bx^*, \lambda^*)$ with $\lambda^*=0$.
\end{proposition}
\begin{proof}
	We reduce the NP-hard subset-sum problem \cite{toth1990knapsack,wojtczak2018strong} to our problem. Consider a set of positive real numbers 
	$\bm{s} = \{ s_1,\ldots, s_n\}$ and   $\gamma > 0$.  
	The subset sum problem is to determining whether there exists a subset $s_i, i\in \Ical $ satisfying 
	$\sum_{i\in \Ical } s_i = \gamma$. Obviously, this problem can be reduced to the problem of determining whether 
	\eqref{eq:general_kkt1} with  $\bar{y}_{i} :=  s_i^{1/p}$, $i \in [n]$   has a 
	solution  $(\bx^*, \lambda^*)$ with $\lambda^*=0$. Hence, to determine whether \eqref{projection.lp.+}  has a 
	first-order stationary point  $(\bx^*, \lambda^*)$ with $\lambda^*=0$ is NP-hard.
\end{proof}

\subsection{Local Complexity Analysis} 
This subsection establishes a $O(1/\sqrt{k})$ convergence rate for the optimality residuals in an ergodic sense.

The proof of the following local complexity result can be found in~\Cref{app_local_complexity}.

\begin{lemma}\label{lem: local_complexity_analysis}
	Assume $0 \not\in \Lambda$. Then there exist $\bar{k}\in\Ucal$ and positive constants $C_\alpha$ and $C_\beta $ such that for all $k \ge \bar{k}$
	$$ \min_{\bar{k} \leq t \leq \bar{k} + k} \alpha(\bm{x}^{t}, \lambda^{t}) \leq \sqrt{C_\alpha/k}  \ \text{ and }\  \min_{\bar{k} \leq t \leq \bar{k} + k} \beta(\bm{x}^{t})  \leq \sqrt{C_\beta /k}.$$
\end{lemma} 

\subsection{Worst-case Complexity Analysis}
The purpose of this subsection is to analyze the worst-case complexity of IRBP. A primal-dual pair $(\tilde\bx,\tilde\lambda)\in\mathbb{R}^{n}\times\mathbb{R}_{+}$ is an $\varepsilon$-optimal solution to \eqref{projection.lp.+} if 
\begin{equation}\label{eps opt} 
	\alpha(\tilde\bx,\tilde\lambda) \le \varepsilon\ \text{ and } \ \beta(\tilde\bx) \le \varepsilon
\end{equation}
with given tolerance $\varepsilon > 0$. 

The iteration complexity of~\Cref{alg.l1} is established in the following theorem.

\begin{theorem}\label{main thm complexity} Given tolerance $\varepsilon > 0$. Let $\bm{\epsilon}\in\mathbb{R}^n_{++}$ satisfy the following
	\begin{equation}\label{choose eps}
		\|\bm{\epsilon}\|_p^p < \gamma,~ p\Vert\bar{\bm{y}}\Vert_{1}(\Vert\bar{\bm{y}}\Vert_{\infty} + \Vert\bm{\epsilon}\Vert_{\infty})^{1-p}\Vert\beps\Vert_{p}^{p}\leq\varepsilon/2 ~\text{and}~   2^p \|\bm{\epsilon}\|_p^p \le \varepsilon/2.
	\end{equation} 
	Assume that $\{(\bx^k, \beps^k)\}$ is generated by \Cref{alg.l1} with
	initial relaxation vector $\beps^0 = \beps$, and the relaxation vector is kept 
	fixed so that $\beps^k = \beps$ for all $k\in \mathbb{N}$.  Then in at most 
	$O(1/\varepsilon^2)$ iterations, 
	$(\bx^k, \beps^k)$ is an $\varepsilon$-optimal solution 
	to \eqref{projection.lp.+}.
\end{theorem}

The proof of such a result depends upon a few preliminary results. To proceed, we first define the relaxed optimal residuals with relaxation vector $\beps$ as 
\begin{equation}\label{relaxed residual}
	\begin{aligned} 
		\alpha_{\beps}(\bx,\lambda) =  &\ \sum_{i=1}^n | ( \bar y_i - x_i)x_i-  \lambda p (x_i+\epsilon_i)^{p-1}x_i|, \\ 
		\beta_{\beps}(\bx) = &\  \left| \sum_{i=1}^{n} (x_{i} + \epsilon_{i})^{p}- \gamma\right|.
	\end{aligned}
\end{equation}

In the following analysis, we first  show that 
$ \alpha_{\beps}(\bx,\lambda) $ and $  \beta_{\beps}(\bx) $ can be used to 
approximate $ \alpha(\bx,\lambda) $ and $  \beta(\bx) $, respectively. Moreover, the approximation errors can be controlled by $\beps$ as shown in the following lemma and its proof is provided in~\Cref{app_D}.
\begin{lemma}\label{Lem: approximation_is_controlled}
	Suppose $(\bx, \bm{\epsilon})$ with $\bm{0}\le \bx\le \bar{\bm{y}}$, $\bm{\epsilon}\in\mathbb{R}^n_{++}$ 
	and $\|\bm{\epsilon}\|_p^p < \gamma$. It then satisfies
	\begin{equation*}\label{eq: residual_dimishes}
		|\alpha(\bx, \lambda)-\alpha_{\beps}(\bx, \lambda)| \le \lambda p  \|\beps\|_p^p \quad\text{and}\quad \vert\beta(\bx) - \beta_{\beps}(\bx)\vert <  2^p \|\bm{\epsilon}\|_p^p.
	\end{equation*}
\end{lemma}

For ease of exposition throughout this subsection, we assume that the hypotheses of~\Cref{main thm complexity} hold. Next, we prove that $\{\alpha_{\beps}(\bx^k,\lambda^k)\}$ and $ \{\beta_{\beps}(\bx^k) \}$ vanish in the following theorem, and the proof  can be found in~\Cref{app_H}.

\begin{lemma} \label{alphabeta}   
	Let the hypotheses of~\Cref{main thm complexity}   hold and $\{(\bx^k,\lambda^k)\}$ be generated by \Cref{alg.l1}. 
	The following statements hold:
	\begin{enumerate}
		\item[(i)] $0 \le  \lambda^k \le  \kappa_3/p$, 
		\item[(ii)] $\alpha_{\beps}(\bx^k,\lambda^k) \le   \kappa_3 \Vert\bar{\bm{y}}\Vert_{2}   (1-p) \|\beps^{p-2}\|_{2} \| \bx^{k-1} - \bx^k\|_2$,
		\item[(iii)]  $\beta_{\beps}(\bx^k)  \le \ p \|\beps^{p-1}\|_2 \|\bx^{k+1}-\bm{x}^k\|_2,$ 
	\end{enumerate} 
	where  $\kappa_3 : = \frac{\Vert\bar{\bm{y}}\Vert_{1}}{ (\Vert\bar{\bm{y}}\Vert_{\infty} + \Vert\bm{\epsilon}\Vert_{\infty})^{p-1}}$. 
\end{lemma}

The following corollary follows directly from Lemma \ref{alphabeta}.
\begin{corollary}
	Given that $\beps$ is a vector of $\epsilon$, i.e., $\beps = (\epsilon,\ldots,\epsilon)^{T}$, we have $\alpha_{\beps}(\bx^k,\lambda^k) \le \frac{\sqrt{n}(1-p)\epsilon^{p-2}\Vert\bar{\bm{y}}\Vert_{1}\Vert\bar{\bm{y}}\Vert_{2}}{(\Vert\bar{\bm{y}}\Vert_{\infty} + \epsilon)^{p-1}}   \| \bx^{k-1} - \bx^k\|_2$.
\end{corollary}

In what follows, we prove the iteration complexity for an approximate critical
point satisfying~\eqref{eps opt}, and the proof can be found in~\Cref{app_I}.
\begin{lemma}\label{Lemma: the maximum iteration}
	Let the hypotheses of~\Cref{main thm complexity}   hold and $\{(\bx^k,\lambda^k)\}$ be generated by \Cref{alg.l1}, and let $\bm{x}^*$ be a cluster point of the sequence $\{\bm{x}^{k}\}_{\mathcal{U}}$. 	Given the tolerance $\varepsilon > 0$, then~\Cref{alg.l1} terminates in finite many iterations to reach $\varepsilon$-optimality, i.e.
	$\max\{\alpha_{\beps}(\bm{x}^{k}, \lambda^k), \beta_{\beps}(\bm{x}^{k})\} < \varepsilon/2$. 
	Furthermore, the maximum iteration number to reach $\varepsilon$-optimality is given by $\max(\kappa_\alpha/\varepsilon^{2},\kappa_\beta/\varepsilon^{2})$ where $\kappa_\alpha$ and $\kappa_\beta$ are constants as defined below
	\begin{equation*}
		\begin{aligned}
			\kappa_\alpha &=  4(1-p)^{2} \kappa_3^2\Vert\bar{\bm{y}}\Vert_{2}^{2} \| \bm{\epsilon}^{p-2}\|_2^{2}(\|\bm{x}^0 - \bar{\bm{y}}\|_{2}^{2} - \|\bm{x}^* - \bar{\bm{y}}\|_2^2), \\
			\kappa_\beta &= 4p^{2} \|\beps^{p-1}\|_{2}^{2}  (\|\bm{x}^0 - \bar{\bm{y}}\|_{2}^{2} - \|\bm{x}^* - \bar{\bm{y}}\|_2^2),
		\end{aligned}
	\end{equation*}
	and $\kappa_3$ is as defined in Lemma \ref{alphabeta}.
\end{lemma}

Now we prove~\Cref{main thm complexity} in the following.
\begin{proof}
	By Lemma \ref{Lemma: the maximum iteration}, IRBP performs at most $O(1/\varepsilon^{2})$ iterations to reach 
	\begin{equation}\label{eq: part1}
		\alpha_{\beps}(\bx^k,\lambda^k)\le \varepsilon/2  \quad\textrm{and}\quad \beta_{\beps}(\bx^k) \le \varepsilon/2.
	\end{equation}
	On the other hand, by Lemma \ref{Lem: approximation_is_controlled}, Lemma \ref{alphabeta} (i) and~\eqref{choose eps}, we have
	\begin{align}
		&|\alpha(\bx^k, \lambda^k)-\alpha_{\beps}(\bx^k, \lambda^k)| \le  \lambda^k p  \|\beps\|_p^p \leq  \kappa_3 \Vert\beps\Vert_{p}^{p}\leq\varepsilon/2, \label{eq: alpha_varepsilon} \\ 
		&\vert\beta(\bx^k) - \beta_{\beps}(\bx^k)\vert <  2^p \|\bm{\epsilon}\|_{p}^{p} \leq \varepsilon/2.\label{eq: beta_varepsilon}
	\end{align}
	
	Combining~\eqref{eq: part1}, ~\eqref{eq: alpha_varepsilon} and~\eqref{eq: beta_varepsilon}, it gives
	\begin{equation}\label{eq: result1}
		\begin{aligned}
			\alpha(\bx^k, \lambda^k) \leq  | \alpha(\bx^k, \lambda^k) -  \alpha_{\beps}(\bx^k, \lambda^k) | +  \alpha_{\beps}(\bx^k, \lambda^k)  
			\le   \varepsilon/2 + \varepsilon/2 = \varepsilon,
		\end{aligned} 
	\end{equation}
	and similarly, 
	\begin{equation}\label{eq: result2}
		\begin{aligned}
			\beta(\bx^k) \leq   | \beta(\bx^k) -  \beta_{\beps}(\bx^k) | +  \beta_{\beps}(\bx^k)  
			\le  \varepsilon/2 + \varepsilon/2 = \varepsilon.
		\end{aligned} 
	\end{equation}
	This completes the proof.
\end{proof}

\section{Numerical Experiments}\label{sec: exprements}
In this section, we conduct a set of numerical experiments on synthetic data  to demonstrate the efficiency of the proposed IRBP for solving the $\ell_{p}$ ball projection problems. The code is implemented in Python.\footnote{We have made the code publicly available at~\url{https://github.com/Optimizater/Lp-ball-Projection}.} and all experiments are performed on an Intel Core CPU i7-7500U at $2.7$GHz$\times 4$ with $7.5$GB of main memory, running under a Ubuntu $64$-bit laptop.
\subsection{Data Description and Implementation Details}
In all tests, we consider projecting $\bm{y}\in\mathbb{R}^{n}$ onto the unit $\ell_{p}$ ball (i.e., $\gamma = 1$). Motivated by the tests on $\ell_1$-ball projections in \cite{condat2016fast}, the projection vector is generated as $\bm{y}\sim N(\gamma/n, 10^{-3})$ by discarding $\bm{y}$ that lies inside the unit $\ell_{p}$ ball.

The input parameters for~\Cref{alg.l1} are set as 
$$\theta = \textrm{min}(\beta(\bm{x}^{k}),1/\sqrt{k})^{1/p},\  \tau = 1.1\ \text{ and } \ M = 100.$$
Meanwhile, we initialize $\bm{x}^{0} = \bm{0}$ and $\bm{\epsilon}^{0} = 0.9(\tfrac{\gamma }{\Vert\bm{\nu}\Vert_{1}}\bm{\nu})^{1/p}$ such that $\Vert\bm{\epsilon}^{0}\Vert_{p}^{p} < \gamma$, where entries of $\bm{\nu}\in \mathbb{R}^{n}$ are
generated from  i.i.d. $U\left[ 0, 1\right] $. The algorithm is terminated when the following condition is satisfied
\begin{equation}\label{eq: practical_termination}
	\max\{\bar{\alpha}(\bm{x}^{k}, \lambda^{k}),\beta(\bm{x}^{k})\} \leq \delta^{\textrm{tol}}\max(\bar{\alpha}(\bm{x}^{0}, \lambda^{0}), \beta(\bm{x}^{0}), 1), 
\end{equation}
where $\bar{\alpha}(\bm{x}^{k},\lambda^{k}) =  \frac{1}{n} {\alpha}(\bm{x}^{k},\lambda^{k})$ and 
$\delta^{\textrm{tol}} = 10^{-8}$. 

To our knowledge, the root-searching procedure (RS for short in the sequel) proposed in \cite{chen2019outlier} was reported to be the most successful algorithm for $\ell_p$-ball projection so far. Therefore, we deliver a set of performance comparison experiments between the RS and the IRBP.
We summarize the RS method in~\Cref{alo: RS}.

On the other hand,  if RS fails to find an optimal multiplier $\lambda^{*}$ satisfying  $\frac{1}{n} | \|\bx\|_p^p - \gamma | < 10^{-8}$, the run is considered as a failure. Similarly, any time IRBP fails to satisfy~\eqref{eq: practical_termination} within $10^{3}$ iterations, we deem this run as a failure.
\subsection{Test Results}

\textbf{First experiment: A $2$-dimension illustrative example.} In this example, we show the iteration path of IRBP.
\begin{example}
	Given $\gamma = 1$, $ p = 0.5$, and $\bm{y} = \left[0.5, 0.45 \right]^{T}$, then the problem~\eqref{projection.lp} takes the form
	\begin{equation}\label{eq: 2d_example}
		\min_{\bm{x}\in\mathbb{R}^{2}}\ \ \frac{1}{2}\left[ (x_1 - 0.5)^{2} + (x_2 - 0.45)^{2}\right]\quad \textrm{\textrm{\normalfont s.t.}}\ \ \sqrt{|x_{1}|} + \sqrt{|x_{2}|} \leq 1.
	\end{equation}
\end{example}

The iteration is started with $(\bm{x}^{0}, \bm{\epsilon}^{0}) = (\left[0, 0\right]^{T}, \left[0.072, 0.463\right]^{T})$, where 
$\bm{\epsilon}^0$ is generated randomly as mentioned above. After $22$ iterations, it converges to one global optimal solution $\bm{x}^{*} = \left[0.2972, 0.2069\right]^{T}$.

We plot the iteration path of IRBP in \Cref{fig:sol_path} using the above initial point. In particular, some of the weighted $\ell_1$-balls in the subproblems are also plotted. Notice that the iterates always remain within the $\ell_p$ ball and eventually move towards the boundary of the ball. 
\begin{figure}[H]
	\centering
	\includegraphics[width=0.7\textwidth]{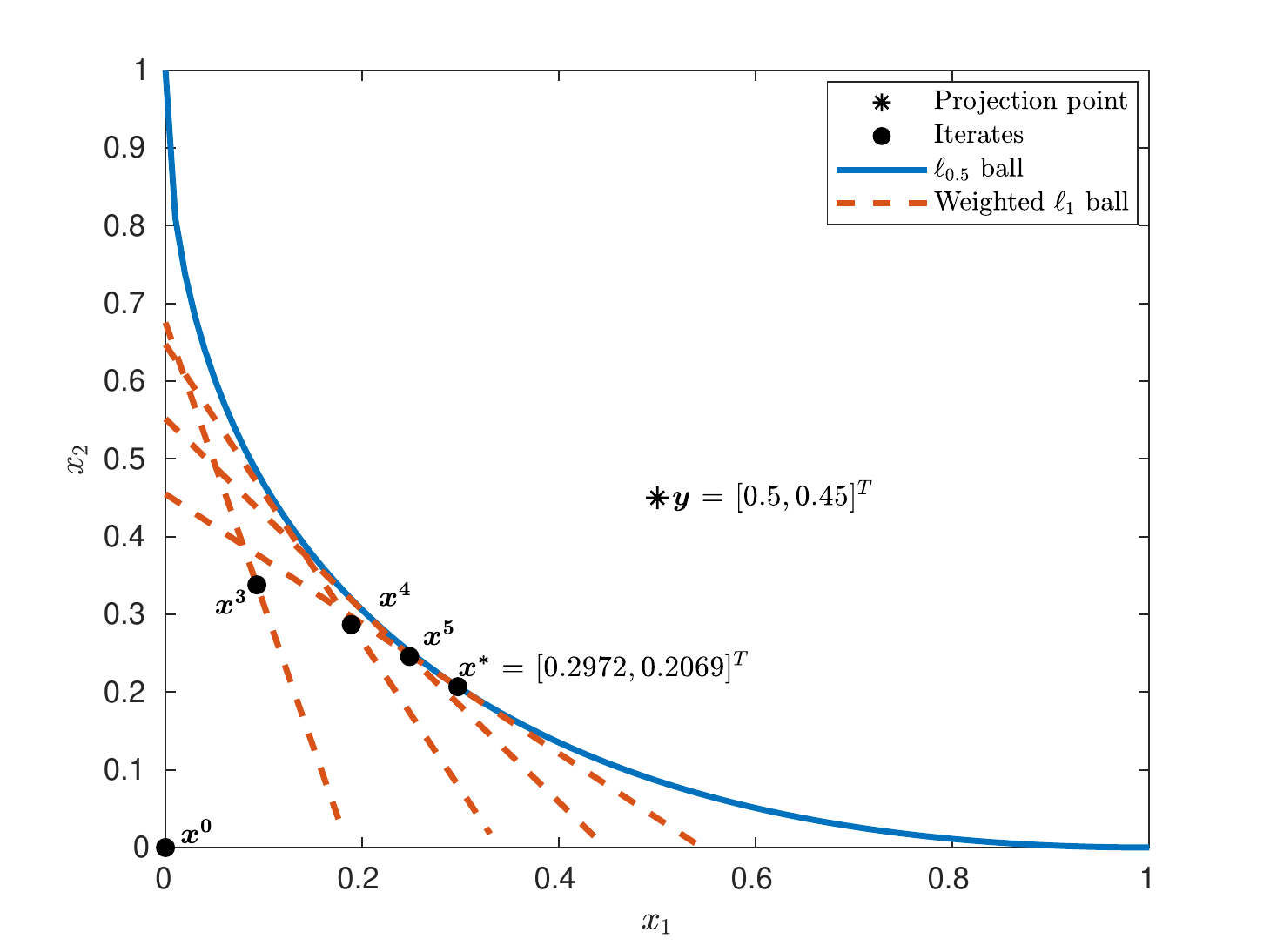}
	\caption{The iteration path of projecting $\bm{y} = \left[0.5, 0.45\right]^{T}$ onto the unit $\ell_{0.5}$-ball. In particular, the solid circles denote the iterates with the randomly generated $\bm{\epsilon}^{0} = \left[0.072, 0.463\right]^{T}$. The dashed line represents the boundary of a weighted $\ell_1$ ball constructed at each subproblem.} 
	\label{fig:sol_path}
\end{figure}

{\textbf{Second experiment: Performance  comparison.}} In this test, we compare the elapsed wall-clock time of IRBP and RS in successful runs. The comparison is done for $n=10^{2}$ and  $p \in \{0.4, 0.8\}$ with $100$ randomly generated test problems for each $p$. 
We compare the success ratio versus the elapsed wall-clock time and plot the performance profiles 
in \Cref{fig:profile}. We make the following observations:
\begin{itemize}
	\item For both cases, we can see that IRBP always achieves a higher success rate compared to RS. In particular, IRBP successfully solves all testing instances for both cases,  whereas RS can only handle about $25\%$ of instances for $p=0.4$ and about $70\%$ for $p=0.8$. On the other hand, both algorithms become more efficient for larger $p$ values.
	
	\item IRBP is superior to RS in terms of elapsed wall-clock time. For most of the test examples, the computational time needed by IRBP is less than $0.1$s for $p=0.4$ and $0.01$s for $p=0.8$. Notice that  the subproblem solver applied here  has an observed complexity $O(n)$ in many practical contexts. We are well aware that the algorithm can be further accelerated if one can design a faster method for the weighted $\ell_1$ ball. 
\end{itemize}
\begin{figure}[H]
	\centering
	\subfigure[$p=0.4$]{\includegraphics[width=0.495\textwidth]{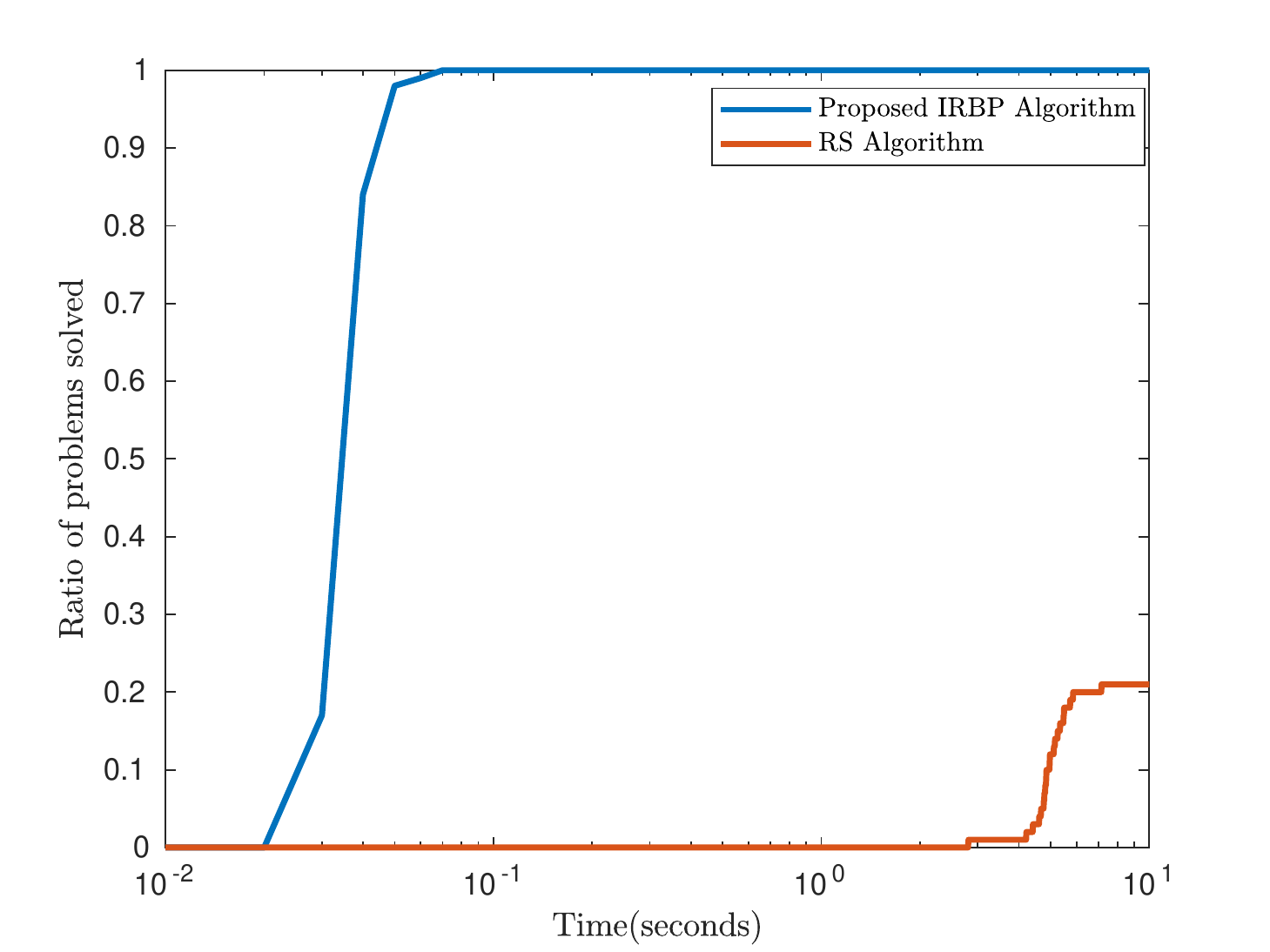}} 
	\subfigure[$p=0.8$]{\includegraphics[width=0.495\textwidth]{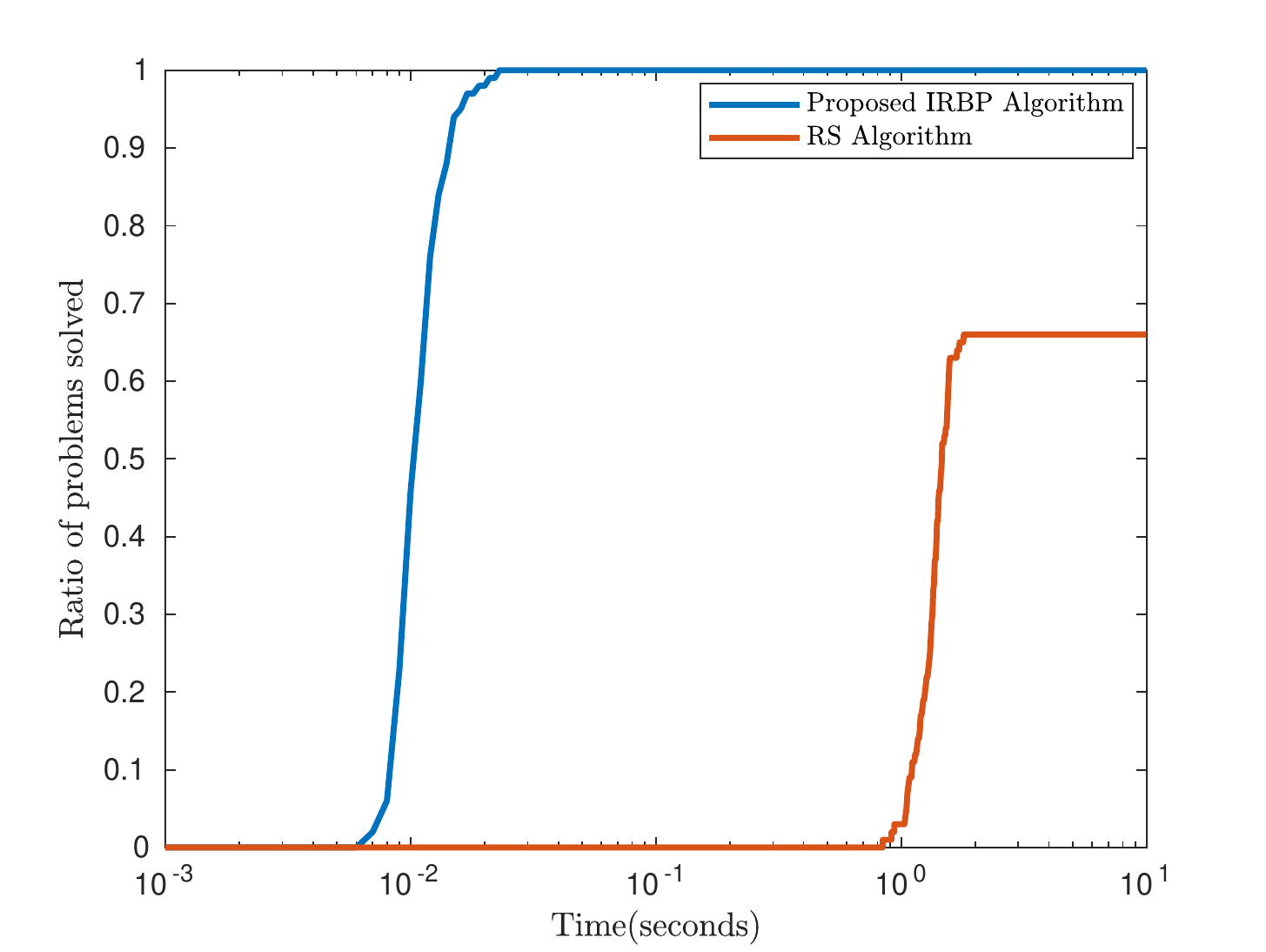}}\\
	\caption{Performance profiles for IRBP and RS with $p \in \{0.4,0.8\}$ and $n = 10^{2}$. }
	\label{fig:profile}
\end{figure}

{\bf Third experiment:  IRBP solves high-dimension problems.} 
In this experiment, we demonstrate  the efficiency of IRBP for solving large-scale problems. The RS algorithm is not included in this experiment since it fails to converge within the time limit. With $p \in  \{0.4,0.8\} $ and $n \in  \{10^{3}, 10^{4}, 10^{5}, 10^{6}\} $.  \Cref{fig: box_plot} reports the computation time, and the results are averaged over $20$ randomly generated test problems.
\begin{figure}[htbp]
	\centering
	\subfigure[$p=0.4$]{\includegraphics[width=0.495\textwidth]{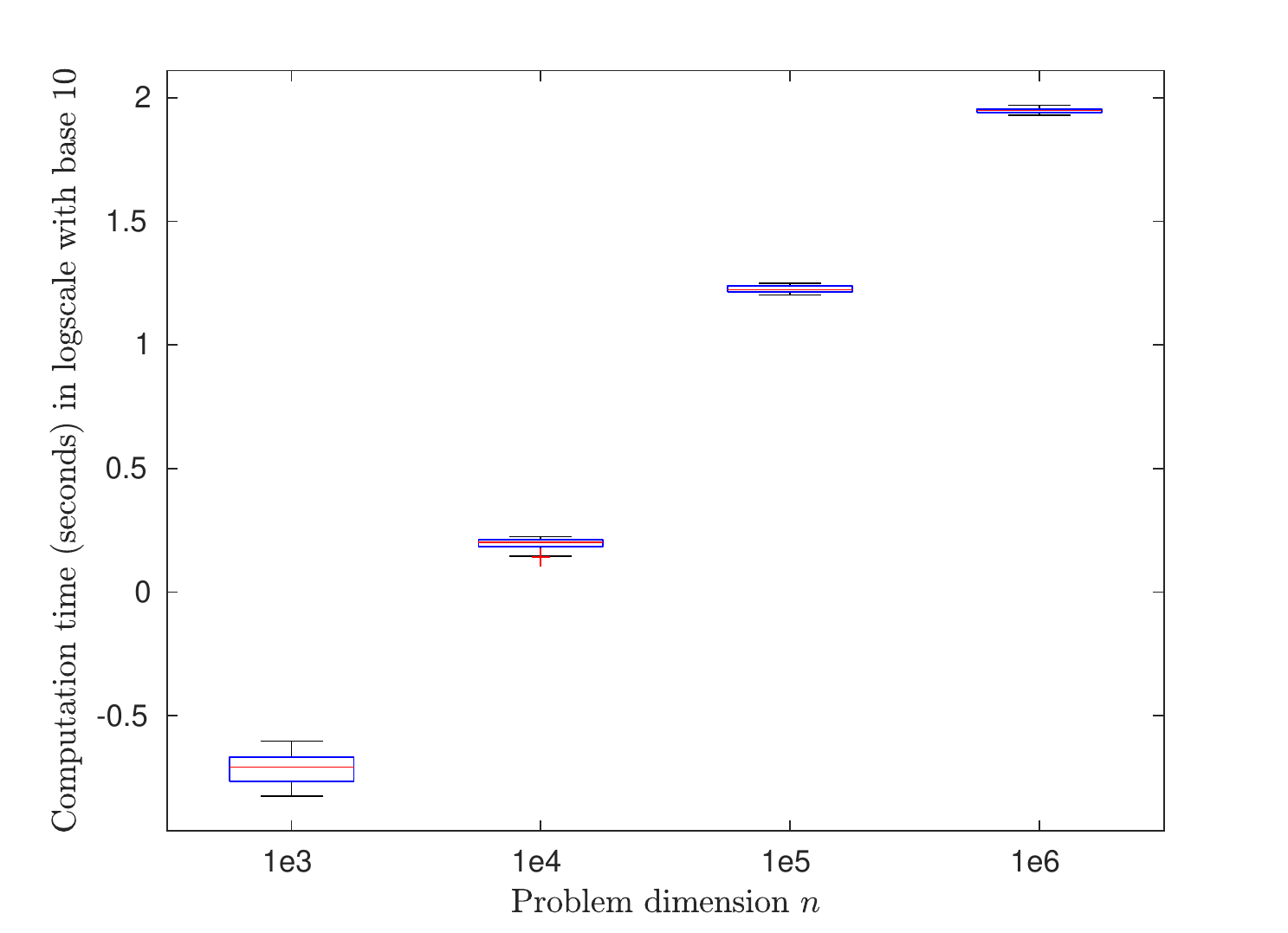}} 
	\subfigure[$p=0.8$]{\includegraphics[width=0.482\textwidth]{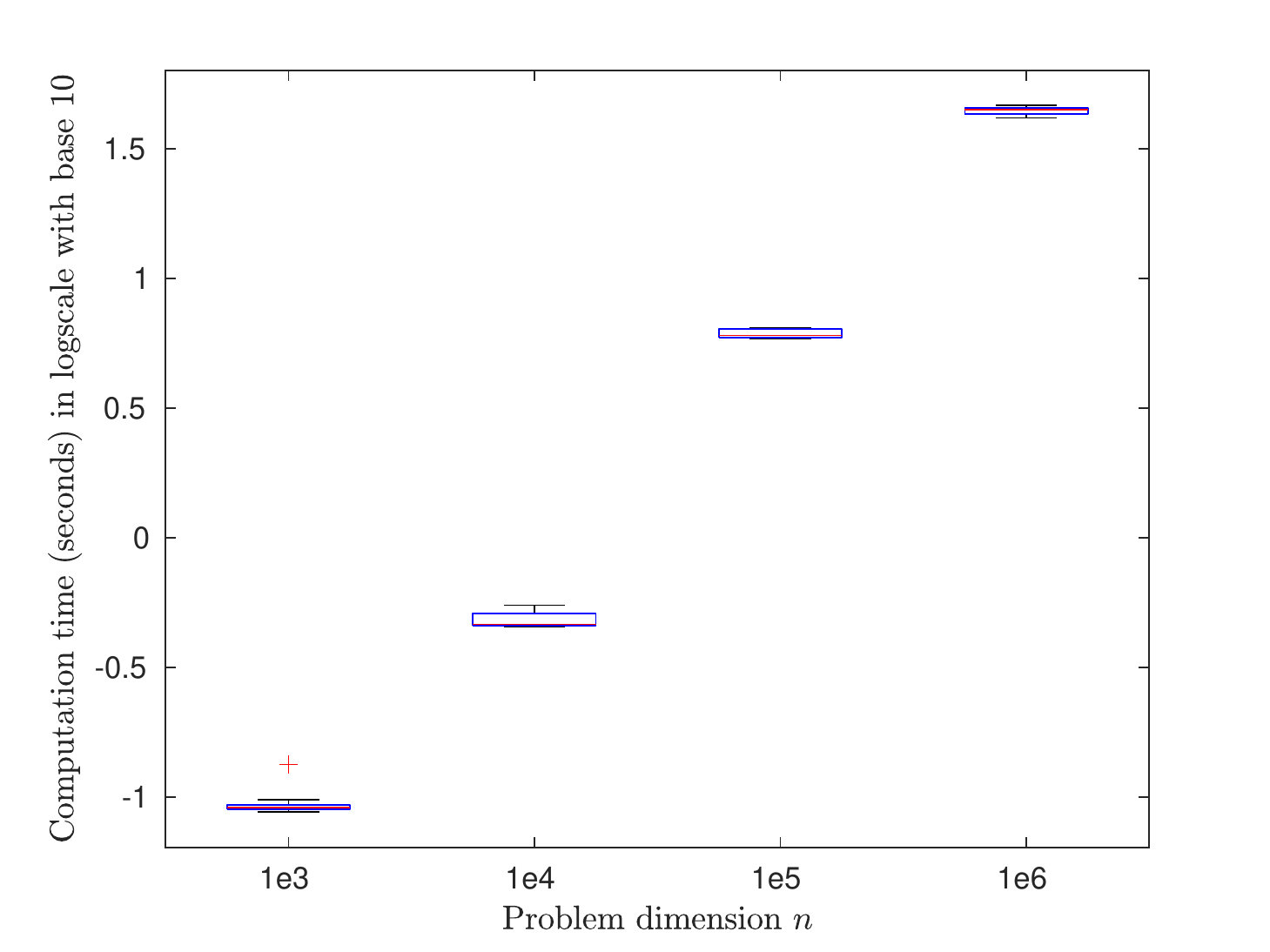}}\\
	\caption{Box plots of the elapsed wall-clock time (seconds) for IRBP with $p \in \{0.4, 0.8\}$ and $n \in \{10^{3}, 10^{4}, 10^{5}, 10^{6}\}$. Each presented elapsed wall-clock time value is averaged over $20$ runs.}
	\label{fig: box_plot}
\end{figure}

{\bf Fourth experiment: Sensitivity to parameters $M$ and $\tau$.} Following the data generation setting at the beginning of this section, in this experiment, we report the performance measures under different combinations of $(M,\tau)$. We use ``Obj'' and ``Time'' to denote the objective function value in~\eqref{projection.lp.+} and the elapsed wall-clock time in implementing IRBP, respectively.  \Cref{tab:Seek_M_and_tau} shows the test results, and the results are averaged over $20$ randomly generated test problems.

\begin{table}[H]
	\centering
	\caption{Performance measures of IRBP under different combinations of $(M,\tau)$ for $p=0.5$ and $n=10^5$.}
	\label{tab:Seek_M_and_tau}
	\begin{tabular}{c|l|c|c|c|c}
		\hline
		$M$                     & \multicolumn{1}{c|}{$\tau$} & Obj     & $\bar{\alpha}$ & $\beta$    & Time(s) \\ \hline
		\multirow{4}{*}{$10$}   & $1.01$                      & 158.145 & 2.0577e-11    & 1.0246e-10 & 17.19   \\
		& $1.1$                       & 158.322 & 2.0464e-11    & 1.0240e-10 & 17.06   \\
		& $1.5$                       & 158.155 & 2.0488e-11    & 1.0281e-10 & 17.08   \\
		& $1.8$                       & 158.012 & 2.1386e-11    & 1.0258e-10 & 16.97   \\ \hline
		\multirow{4}{*}{$10^2$} & $1.01$                      & 158.092 & 2.0793e-11    & 1.0298e-10 & 17.43   \\
		& $1.1$                       & 158.130 & 2.0630e-11    & 1.0223e-10 & 17.70   \\
		& $1.5$                       & 158.152 & 2.0947e-11    & 1.0256e-10 & 17.07   \\
		& $1.8$                       & 157.956 & 2.0934e-11    & 1.0245e-10 & 17.38   \\ \hline
		\multirow{4}{*}{$10^3$} & $1.01$                      & 158.189 & 2.1201e-11    & 1.0279e-10 & 17.21   \\
		& $1.1$                       & 158.179 & 2.1532e-11    & 1.0245e-10 & 17.56   \\
		& $1.5$                       & 158.244 & 2.0938e-11    & 1.0298e-10 & 17.43   \\
		& $1.8$                       & 158.149 & 2.0642e-11    & 1.0257e-10 & 17.19   \\ \hline
		\multirow{4}{*}{$10^4$} & $1.01$                      & 158.140 & 2.0456e-11    & 1.0252e-10 & 17.17   \\
		& $1.1$                       & 158.076 & 2.1416e-11    & 1.0253e-10 & 17.13   \\
		& $1.5$                       & 158.024 & 2.1136e-11    & 1.0233e-10 & 17.21   \\
		& $1.8$                       & 158.245 & 2.0674e-11    & 1.0224e-10 & 17.06   \\ \hline
	\end{tabular}
\end{table}

From~\Cref{tab:Seek_M_and_tau}, we can see that the IRBP is generally not sensitive to the choice of $M$ and $\tau$. Based on our experience, we suggest that the choice of $\tau$ should be appropriately close to $1$, and $M$ could be determined by rough tuning depending on the problems at hand.

\section{Conclusion}\label{sec: Conclusion}
In this paper, we have proposed, analyzed, and tested an iteratively reweighted $\ell_1$-ball projection approach for solving the projection of a given vector onto the $\ell_{p}$ ball. The central idea of the algorithm is to perform projections onto a sequence of simple and tractable weighted $\ell_{1}$-balls. First-order necessary optimality conditions are derived. We have established global convergence and analyzed the worst-case complexity of the proposed algorithm. Numerical experiments have demonstrated the effectiveness and efficiency of the algorithm. 

We believe that our work could lead to a breakthrough movement in the research of solving $\ell_p$-ball constrained optimization problems. With a practical $\ell_p$-ball projection method,  many intractable $\ell_p$-ball constrained optimization problems can now be handled. Future work may involve analyzing the local convergence rate,  accelerating the $\ell_p$-ball projection algorithm, and exploring the applicability of this problem in many applications.


\appendix
\section{Proofs of technical propositions and lemmas}

\subsection{Proof of~\eqref{eq.tmp.xhat}}\label{app_A}
\begin{proof}
	To compute the optimal solution $\hat{\bm{x}}$, we first derive the optimality condition of its associated optimization problem in the following
	\[
	(2 + 2\nu)\hat{\bm{x}} = 2\nu \bar{\bm{y}},\quad  \nu \geq 0,
	\]
	where $\nu$ is the defined nonnegative multiplier. It follows that 
	$\hat{\bm{x}} = \frac{\nu}{1+\nu}\bar{\bm{y}}.
	$
	Therefore, it satisfies $\Vert(\frac{\nu}{1+\nu} - 1)\bar{\bm{y}}\Vert_{2}^{2} \leq \Vert\bm{x}^{\bar{k}} - \bar{\bm{y}}\Vert_{2}^2$. 
	Simple rearrangement yields that 
	$
	\nu \geq \frac{\Vert\bar{\bm{y}}\Vert_{2}}{\Vert\bm{x}^{\bar{k}} - \bar{\bm{y}}\Vert_{2}} -1.
	$
	Since $\nu \geq 0$ and $\bm{x}^{\bar{k}} \neq \bm{0}$, we have $0 < \Vert\hat{\bm{x}}\Vert_{2} \leq\Vert\bar{\bm{y}}\Vert_2 - \Vert\bm{x}^{\bar{k}} - \bar{\bm{y}}\Vert_{2}$, completing the proof.
\end{proof}
\subsection{Proof of Lemma \ref{lem.eps to 0}}\label{app_Trigger-Inf}
\begin{proof} 
	Suppose by contradiction that there exists $\hat k \in \Nb$ such that 
	\begin{equation}\label{temp.eq.1}
		\|\bm{x}^{k+1} - \bm{x}^k\|_2\|\text{sign}(\bx^{k+1} - \bx^k) \circ \bw^{k}\|_{2}^{  \tau } >  M,\  \textrm{ for all }  k \ge \hat k,
	\end{equation}	
	Since $\bm{\epsilon}$ is never decreased after iteration $\hat k$,  it holds that 
	\[ w_i^k = p(|x_i^k|+\epsilon_i^k)^{p-1} \le p(\epsilon_i^{\hat k})^{p-1}\]
	for 
	all $i\in[n]$. 
	Thus, it follows from \eqref{temp.eq.1} that 
	\begin{equation*}
		\|\bm{x}^{k+1}-\bm{x}^k\|_2  > M \| \text{sign}(\bx^{k+1} - \bx^k) \circ  \bw^{k}\|_2^{-\tau} \ge M\big( \sum\limits_{x_i^{k+1} \ne x_i^k} p^2(\epsilon_i^{\hat k})^{2(p-1)}  \big)^{-\tau/2} > 0
	\end{equation*} 
	for all $k \ge \hat k$, contradicting Lemma \ref{lem.x-x} (ii). This completes the proof.
\end{proof}
\subsection{Proof of Proposition \ref{lem.stable_nonzero}}\label{app_3importantProperty}
\begin{proof} 
	Since $0 \not\in \Lambda_{\Ucal}$, there exists $\hat{\lambda}$ such that $\lambda^{k} \geq \hat{\lambda}$ for all $k\in\Ucal$. 
	\begin{itemize}
		
		\item[(i)] Consider a subsequence $\mathcal{S}\subset \Ucal$. There exists $\delta_i > 0 $ such that for sufficiently large $k\in \mathcal{S} \subset \Ucal$, $x_{i}^{k} > \delta_{i} > 0$ for all $i \in\Ical^k$ (notice that $\Ical^{k} \neq \emptyset$  by Lemma \ref{lemma.primal}(i)). 
		
		We prove this by contradiction and suppose this is not true. By Lemma \ref{lem.x-x} and Lemma \ref{lem.eps to 0}, there exists $\bar{k}$ such that $|x_i^{k} - x_i^{k-1}| < \tfrac{1}{4} {\big(  \tfrac{1}{\hat{\lambda} p}  \|\bar{\bm{y}}\|_\infty \big)^{1-p}} $ and $|\epsilon_{i}^{k} - \epsilon_{i}^{k-1}|< \tfrac{1}{4} {\big(  \tfrac{1}{\hat{\lambda} p}  \|\bar{\bm{y}}\|_\infty \big)^{1-p}} $ for all $k > \bar{k}$ and $i\in\Ical^{k}$. On the other hand, by assumption, there exists $\hat{k}$ such that $x_{i}^{k} < \tfrac{1}{4} {\big(  \tfrac{1}{\hat{\lambda} p} \|\bar{\bm{y}}\|_\infty \big)^{1-p}}  $ and $\epsilon_{i}^{k} <  \tfrac{1}{4} {\big(  \tfrac{1}{\hat{\lambda} p}  \|\bar{\bm{y}}\|_\infty \big)^{1-p}} $ for all $k > \hat{k}$ and $i\in\Ical^{k}$. Hence, for all $k > \max(\bar{k},\hat{k})$, this implies
		\begin{equation*}
			\begin{aligned}
				\bar{y}_{i} - x_{i}^{k} = \lambda^{k}p(x_{i}^{k-1} + \epsilon_{i}^{k-1})^{p-1} &\geq \hat{\lambda}p(x_{i}^{k-1} + \epsilon_{i}^{k-1})^{p-1} \\
				& \geq \hat{\lambda}p (x_i^k + |x_i^{k} - x_i^{k-1}| + \epsilon_{i}^k + |\epsilon_{i}^{k} - \epsilon_{i}^{k-1}|)^{p-1}\\
				&>  \|\bar{\bm{y}}\|_\infty.
			\end{aligned}
		\end{equation*}
		
		This contradicts the assumption that $x_{i}^{k} > 0$. 
		Therefore,   $x_{i}^{k} > \delta_i > 0$ for sufficiently large  $k\in \Ucal$ and $i \in \Ical^k$.
		
		Next, we show that $\Ical^k$ and $\Acal^k$ are fixed for sufficiently large $k$. For sufficiently large $k \in \Ucal$ and $i \in \Ical^{k}$, we have $x_{i}^{k+1} > \delta_i / 2 > 0$ since $x_{i}^{k+1} - x_{i}^{k} \to 0$ by Lemma \ref{lem.x-x}. Hence 
		\begin{equation}\label{eq: inactive_set_sub}
			\Ical^k \subset \Ical^{k+1}. 
		\end{equation}
		It then suffices to show $\Acal^{k} \subset \Acal^{k+1}$. 
		For sufficiently large $k \in \Ucal$ and $i \in \Acal^{k}$, we have
		$\lambda^{k+1}w_{i}^{k} = \lambda^{k+1}p(|x_{i}^{k}| + \epsilon_i^k)^{p-1} = \lambda^{k+1}p(\epsilon_i^k)^{p-1} \geq \hat{\lambda}p(\epsilon_i^k)^{p-1}\geq \bar{y}_{i}$, which means  $x_{i}^{k+1} = 0$ by \eqref{kkt.2.l1}. In other words, we obtain 
		$\Acal^{k} \subset \Acal^{k+1}$	for sufficiently large $k\in \Ucal$.
		This, together with~\eqref{eq: inactive_set_sub}, gives
		\begin{equation*}
			\Acal(\bm{x}^{k}) = \Acal(\bm{x}^{k+1})\ \textrm{ and }\ \Ical(\bm{x}^{k}) = \Ical(\bm{x}^{k+1}).
		\end{equation*}
		
		Now, we are ready to show the update condition in~\eqref{update.condition} is also satisfied at the $(k+1)$th iteration. Since $\Ical^{k} \neq \emptyset$ (by Lemma \ref{lemma.primal} (i)) and $\lim\limits_{\substack{ k \to \infty \\ k\in \Ucal  }}\Vert\bm{x}^{k+1}_{\Ical^{k}} - \bm{x}^{k}_{\Ical^{k}}\Vert_{2} = 0$, we know for sufficiently large $k$, 
		\begin{equation*}
			\Vert\bm{w}_{\Ical^{k}}^{k}\Vert_{2}^{2} = p^2 \sum_{i\in\Ical^k}(x_i^{k} + \epsilon_{i}^{k})^{2(p-1)} \leq p^2\sum_{i\in\Ical^k} \delta^{2(p-1)}  
			= p^2 |\Ical^k| \delta^{2(p-1)}.
		\end{equation*}
		Overall, we have  for sufficiently large $k\in\Ucal$,
		\begin{equation}\label{eq: check_update}
			\begin{aligned}
				\Vert\bm{x}^{k+1} - \bm{x}^{k}\Vert_2\Vert\textrm{sign}(\bm{x}^{k+1}-\bm{x}^{k})\circ \bm{w}^{k}\Vert_{2}^{\tau} \leq \Vert\bm{x}^{k+1}_{\Ical^{k}} - \bm{x}^{k}_{\Ical^{k}}\Vert_2\Vert\bm{w}^{k}_{\Ical^{k}}\Vert_{2}^{\tau} < M,
			\end{aligned}
		\end{equation}
		implying $k+1 \in \Ucal$. By induction, we know the condition~\eqref{update.condition} is triggered after $k > \bar{k}$.
		
		\item[(ii)] By statement~(i), we know that $0$ is not a cluster point of $\{\lambda^{k}\}$. Then there exists $\hat{\lambda}$ such that $\lambda^{k} \geq \hat{\lambda}$ for all $k\in\mathbb{N}$. Let $C := \Vert\bar{\bm{y}}\Vert_{\infty}/\hat{\lambda}$.
		
		Suppose $\{w_i^k\}$ is unbounded, then there exists $\bar k$ such that $w_i^k \ge C$. 
		By \eqref{kkt.1.l1}, it holds true that $x_i^{\bar{k}+1}= 0$, yielding 
		\begin{equation*}
			\begin{aligned}
				w_i^{\bar{k}+1} = p(|x_i^{\bar{k}+1}| + \epsilon_i^{\bar{k}+1})^{p-1} = p(\epsilon_i^{\bar{k}+1})^{p-1} \geq p(|x_i^{\bar{k}}| + \epsilon_i^{\bar{k}})^{p-1} = w_i^{\bar{k}},
			\end{aligned}
		\end{equation*}
		where the inequality holds because $(\cdot)^{p-1}$ is nonincreasing and $\epsilon_i^{\bar{k}+1}\leq |x_i^{\bar{k}}| + \epsilon_i^{\bar{k}}$. Hence, by induction, it follows that $x_i^k  \equiv 0$ for any $k > \bar{k}$. On the other hand, suppose $\{w_i^k\}$ is bounded above. By statement~(i), $\{(\bm{x}^{k}, \lambda^{k})\}_{\Ucal}$ coincides with $\{(\bm{x}^{k}, \lambda^{k})\}$ after $k >\bar{k}$,  we then know $x_{i}^{k}$ is strictly bounded away from $0$ for all sufficiently large $k$ . Therefore,  there exists  $\delta_i > 0$
		such that $x_i^k > \delta_{i} >0$
		for all sufficiently large $k$.
		\item[(iii)] This statement follows straightforwardly from~(ii).
	\end{itemize}
	
\end{proof}
\subsection{Proof of Lemma \ref{lem.residual bounded}}\label{app_C}
\begin{proof}
	To derive \eqref{residual.bound1}, recalling \eqref{kkt.1.l1}, we have  
	\begin{equation*} 
		\begin{aligned}
			\alpha(\bm{x}^{k},\lambda^{k})
			=\sum_{i=1}^n |(\bar{y}_i - x_i^k -\lambda^{k} p(x_i^k )^{p-1})x_i^k| &= \lambda^{k}p \sum_{i\in\Ical^k}| (x_i^{k-1}+\epsilon_i^{k-1})^{p-1} (x_i^k )  -  (x_i^k )^p|\\
			&=\lambda^{k} p\sum_{i\in\Ical^k}| (x_i^{k-1} +\epsilon_i^{k-1})^{p-1} - (x_i^k )^{p-1}| |x_i^k|. 
		\end{aligned}
	\end{equation*}
	From the Lagrange's mean value theorem, we have for $i\in\Ical^k$ 
	\begin{equation}\label{lag.mean}
		(x_i^{k-1} +\epsilon_i^{k-1})^{p-1} - (x_i^k )^{p-1} = (p-1) (\tilde{x}_i^k )^{p-2} (x_i^{k-1} + \epsilon_i^{k-1} -x_i^k )
	\end{equation}
	with $\tilde{x}_{i}^{k}$ between $x_i^k$ and $x_i^{k-1}+\epsilon_i^{k-1}$.   It then follows that 
	\begin{equation*} 
		\begin{aligned}
			\alpha(\bm{x}^{k},\lambda^{k})
			\le & \  \lambda^{k} (1-p)p  \sum_{i\in\Ical^k} (\tilde{x}_i^k )^{p-2} x_i^k |x_i^{k-1} + \epsilon_i^{k-1}  -x_i^k | \\
			\le & \  \lambda^{k} (1-p)p [ \sum_{i\in\Ical^k} (\tilde{x}_i^k )^{p-2} x_i^k (|x_i^{k-1}-x_i^k | +  \epsilon_i^k) ]  \\
			\le & \  \lambda^{k} (1-p)p \Vert[\tilde{\bm{x}}^{k}_{\Ical^k}]^{p-2} \Vert_{2}\Vert\bar{\bm{y}}\Vert_{\infty}\left[ \Vert\bm{x}^{k-1}_{\Ical^k} - \bm{x}^{k}_{\Ical^k}\Vert_{2} +\Vert\bm{\epsilon}^k_{\Ical^k}\Vert_{2}\right],
		\end{aligned}
	\end{equation*}
	completing the proof for \eqref{residual.bound1}. 
	
	To derive \eqref{residual.bound2}, 
	we have from Lemma \ref{lem.x-x} (i)
	\begin{equation*} 
		\begin{aligned}
			\beta(\bm{x}^k)  
			= &\ |\sum_{i=1}^{n} (x_i^k)^{p} - \gamma|\\
			\leq & \ |\sum_{i=1}^n (x_i^k)^{p}  - \sum_{i=1}^n[ (x_i^{k-1} +\epsilon_i^{k-1} )^{p} + w_i^{k-1}(x_i^{k} - x_i^{k-1})] |\\ 
			\leq & \ \sum_{i\in\Ical^{k}  } w_i^{k-1}|x_i^{k-1} - x_i^k| +  \sum_{i\in\Ical^{k} } |(x_i^{k-1} +\epsilon_i^{k-1} )^{p}  - (x_i^k)^{p}  |  +\sum_{i\in\Acal^{k}} ( x_{i}^{k-1} 
			+  \epsilon_i^{k-1} )^{p}.
		\end{aligned}
	\end{equation*}
	From the Lagrange's mean value theorem, we have for $i\in \Ical^{k}$, 
	\begin{equation}\label{lag.mean2}
		(x_i^{k-1} +\epsilon_i^{k-1})^p - (x_i^k )^p= p(\hat{x}_i^k )^{p-1} (x_i^{k-1} + \epsilon_i^{k-1} -x_i^k )
	\end{equation}
	with $\hat{x}_i^k $ between $x_i^k $ and $x_i^{k-1}+\epsilon_i^{k-1}$.  
	It then follows that 
	\begin{equation*}
		\begin{aligned}
			\beta(\bm{x}^k)
			\leq & \ \|\bw_{\Ical^k}^{k-1} \|_2 \Vert\bm{x}^{k-1}_{\Ical^{k}} - \bm{x}^{k}_{\Ical^{k}}\Vert_{2} \!+\! p\sum_{i\in\Ical^{k}}(\hat{x}_i^{k})^{p-1}(|x_i^{k-1}-x_i^k|
			+ |\epsilon_i^{k-1} |) \\
			& \ +  \sum_{i\in\Acal^{k}} ( x_{i}^{k-1} 
			+  \epsilon_i^{k-1} )^{p-1}( x_{i}^{k-1} 
			+  \epsilon_i^{k-1} )\\
			\leq & \ \|\bw_{\Ical^k}^{k-1} \|_2 \Vert\bm{x}^{k-1}_{\Ical^{k}} - \bm{x}^{k}_{\Ical^{k}}\Vert_{2} \!+\! p\sum_{i\in\Ical^{k}}(\hat{x}_i^{k})^{p-1}(|x_i^{k-1}-x_i^k|
			+ |\epsilon_i^{k-1} |) \\
			& \ +  \sum_{i\in\Acal^{k}} (\epsilon_i^{k-1} )^{p-1}( x_{i}^{k-1} 
			+  \epsilon_i^{k-1} )\\
			\leq & \  (\|\bw_{\Ical^{k}}^{k-1} \|_2\!+\!p\|[\hat\bx_{\Ical^{k}}^k]^{p-1} \|_2) \Vert\bm{x}^{k-1}_{\Ical^{k}} \!-\! \bm{x}^{k}_{\Ical^{k}}\Vert_{2} +\! p\|[\hat\bx_{\Ical^{k}}^k]^{p-1} \|_\infty  \|\bm{\epsilon}_{\Ical^{k}}^{k-1} \|_1\\ 
			&\ + \|[\beps_{\! \Acal^{k}}^{k-1}]^p\|_1 + \|[\beps_{\! \Acal^{k}}^{k-1}]^{p-1} \circ {\bm{x}}_{\Acal^{k}}^{k-1}\|_{1},
		\end{aligned}
	\end{equation*}
	completing the proof for \eqref{residual.bound2}. 
\end{proof}
\subsection{Proof of~\Cref{Theo: unique_convergence}}\label{app_unique_convergence}
\begin{proof}
	Suppose $\{\bx^k\}$ has multiple cluster points, then $\Gamma$ contains more than one element. Lemma \ref{lemma.primal} (ii) implies  $\bm{x}^*\neq 0$. Hence $\Ical^*\neq \emptyset$. 
	
	By Proposition \ref{lem.stable_nonzero}, for any $(\bm{x}^{*}, \lambda^*) \in \Gamma$ satisfying~\eqref{kkt.2} implies that $\bm{x}_{\Ical^{*}}^{*}$ is a first-order stationary solution to the smoothed $\ell_{p}$ regularization problem in the reduced subspace $\mathbb{R}^{|\Ical^*|}$. Note that the Lagrangian for~\eqref{projection.lp.+} in $\mathbb{R}^{|\Ical^*|}$ is written as 
	\begin{equation}\label{eq: Lagrangian}
		\mathcal{L} (\bm{x}_{\Ical^*},\lambda^{*}) = \frac{1}{2} \Vert\bm{x}_{\Ical^*} - \bar{\bm{y}}_{\Ical^*}\Vert_{2}^2 + \lambda^{*}(\Vert\bm{x}_{\Ical^*}\Vert_{p}^p - \gamma).
	\end{equation}
	Therefore, we have $[\nabla_{\bm{x}_{\Ical^*}}\mathcal{L} (\bm{x}_{\Ical^*}^*,\lambda^*)]_{i} = x_{i}^* - \bar{y}_{i} +  \lambda^* p(x_{i}^*)^{p-1} = 0$ for $i\in \Ical^*$ and $(\bm{x}^*, \lambda^*)\in \Gamma$, meaning 
	every limit point is stationary for $\mathcal{L} (\bm{x}_{\Ical^*},\lambda^*)$.  On the other hand, by \cite[Proposition 6.4]{golubitsky2012stable}, if the Hessian of $\mathcal{L}$ with respect to $\bm{x}_{\Ical^*}$
	\begin{equation*}\label{eq: Hession_reduced}	
		\left[ \nabla^{2}\mathcal{L} (\bm{x}_{\Ical^*},\lambda^*)\right]_{\Ical^*\Ical^*}   = \bm{I}_{\Ical^*\Ical^*} + \text{Diag}(\lambda^* p(p-1)(\bm{x}_{\Ical^*}^*)^{p-2}) 
	\end{equation*}
	is nonsingular, or equivalently, $x_i^* \neq [\lambda^* p(1-p)]^{\frac{1}{2-p}}$ for $i \in \Ical^*$. Then $\bm{x}^{*}$ is an isolated critical point. This contradicts Lemma \ref{prop.property of limits} and the assumption. Hence, $\bm{x}^{*}$ is the unique cluster point of $\{\bm{x}^{k}\}$.
\end{proof}
\subsection{Proof of~\Cref{lem: local_complexity_analysis}}\label{app_local_complexity}
\begin{proof}
	By Proposition \ref{lem.stable_nonzero}, we can assume the $\bar{k}$ gives us the stable $\Ical(\bx^k) = \Ical(\bx^*)$ and $\Acal(x^k) = \Acal(x^*)$ for $k \ge \bar{k}$. Lemma \ref{lem.residual bounded} implies $\{\| [\tilde{\bx}^k_{\Ical^*}]^{p-2}\|_2 \}$ is bounded above by a constant $c_1$.
	
	Let $\bar{\lambda}$ denote an upper bound for $\{\lambda^{k}\}$ by Lemma \ref{lemma.primal} (iv) and define  $\kappa_1:= (\bar{\lambda}p(1-p)c_1\Vert\bar{\bm{y}}\Vert_{\infty}$.
	Then, we rewrite~\eqref{residual.bound1} as
	\begin{equation}\label{eq: rewrite_alpha}
		\begin{aligned}
			\alpha^{2}(\bm{x}^{k}, \lambda^{k}) \le & \  (\kappa_1\Vert\bm{x}_{\Ical^*}^{k-1} - \bm{x}_{\Ical^*}^{k}\Vert_{2} + \kappa_1\Vert\beps_{\Ical^*}^{k}\Vert_{2})^2
			\leq  2\kappa_1^2(\Vert\bm{x}_{\Ical^*}^{k-1} - \bm{x}_{\Ical^*}^{k}\Vert_{2}^{2} + \Vert\beps_{\Ical^*}^{k}\Vert_{2}^{2})\\
			\leq & 2\kappa_1^2(\Vert\bm{x}^{k-1}_{\Ical^*} - \bar{\bm{y}}_{\Ical^*}\Vert_{2}^{2} - \Vert\bm{x}^{k}_{\Ical^*} - \bar{\bm{y}}_{\Ical^*}\Vert_{2}^{2} + \Vert\beps_{\Ical^*}^{k}\Vert_{2}^{2}),
		\end{aligned}
	\end{equation}
	where the second inequality follows from the Cauchy-Schwarz inequality and the last inequality holds due to~\eqref{eq: sum1}.
	
	Replacing $k$ by $t$ and summing up both sides of~\eqref{eq: rewrite_alpha} from $ \bar{k} $ to $ \bar{k}   + k$ yields 
	\begin{equation*}
		\begin{aligned}
			k\min_{\bar{k} \leq t \leq \bar{k} + k} \alpha^{2}(\bm{x}^{t}, \lambda^{t})  &  \leq
			\sum_{t = \bar{k}}^{\bar{k}   + k}\alpha^{2}(\bm{x}^{t}, \lambda^{t}) \le  \sum_{t = \bar{k}}^{\bar{k} + k}2\kappa_1^2(\Vert\bm{x}^{t-1}_{\Ical^*} - \bar{\bm{y}}_{\Ical^*}\Vert_{2}^{2} - \Vert\bm{x}^{t}_{\Ical^*} - \bar{\bm{y}}_{\Ical^*}\Vert_{2}^{2} + \Vert\beps_{\Ical^*}^{t}\Vert_{2}^{2})\\
			&\le  \  2\kappa_1^2(\Vert\bm{x}^{\bar{k}-1}_{\Ical^*} - \bar{\bm{y}}_{\Ical^*}\Vert_{2}^{2} - \Vert\bm{x}^{\bar{k} + k}_{\Ical^*} - \bar{\bm{y}}_{\Ical^*}\Vert_{2}^{2} + \sum_{t = \bar{k}}^{\bar{k} + k}\Vert\beps_{\Ical^*}^{t}\Vert_{2}^{2})\\
			&\le  2\kappa_1^2(\Vert\bm{x}^{\bar{k}-1}_{\Ical^*} - \bar{\bm{y}}_{\Ical^*}\Vert_{2}^{2} - \Vert\bm{x}^{*}_{\Ical^*} - \bar{\bm{y}}_{\Ical^*}\Vert_{2}^{2} + \frac{\theta^2 -\theta^{2(k+1)}}{1-\theta^2}\Vert\beps_{\Ical^*}^{\bar{k}}\Vert_{2}^{2})\\
			& \leq C_\alpha  : =  2\kappa_1^2(\Vert\bm{x}^{\bar{k}-1}_{\Ical^*} - \bar{\bm{y}}_{\Ical^*}\Vert_{2}^{2} - \Vert\bm{x}^{*}_{\Ical^*} - \bar{\bm{y}}_{\Ical^*}\Vert_{2}^{2} + \tfrac{\theta^2  }{1-\theta^2}\Vert\beps_{\Ical^*}^{\bar{k}}\Vert_{2}^{2}),
		\end{aligned}
	\end{equation*}
	where the second inequality follows from~\eqref{eq: sum1} and the last inequality holds because the update of $\beps$ in~\eqref{update.condition} is always triggered. 
	Dividing both sides  by $k$ yields 
	$\min_{\bar{k} \leq t \leq \bar{k} + k} \alpha(\bm{x}^{t}, \lambda^{t}) \leq  \sqrt{C_\alpha / k}$.

	
	Similar  argument  applied to \eqref{residual.bound2} yields  
	\begin{equation*}
		\begin{aligned}
			\ \beta^{2}(\bm{x})
			\leq \kappa_2 (\Vert\bm{x}^{k-1}_{\Ical^*} - \bar{\bm{y}}_{\Ical^*}\Vert_{2}^{2} - \Vert\bm{x}^{k}_{\Ical^*} - \bar{\bm{y}}_{\Ical^*}\Vert_{2}^{2} + \Vert\beps_{\Ical^*}^{k-1}\Vert_{1}^2 + \Vert[\beps_{\Ical^*}^{k-1}]^{p}\Vert_{1}^2)
		\end{aligned}
	\end{equation*} 
	with $\kappa_2 :=( \bar{w} + pc_2)^2 + p^2c_2^2 + 1$. 
	Replacing $k$ by $t$ and summing both sides   from $\bar k$ to $\bar{k}  + k$ yields that 
	\begin{equation*} 
		\begin{aligned}
			& k\min_{\bar{k} \leq t \leq\bar{k} + k} \beta^{2}(\bm{x}^{t} ) 	\leq  
			\sum_{t =\bar{k} }^{\bar{k} +k}\beta^{2}(\bm{x}^{t}) \\
			\leq & \kappa_2 (\Vert\bm{x}^{\bar{k}-1}_{\Ical^*} - \bar{\bm{y}}_{\Ical^*}\Vert_{2}^{2} - \Vert\bm{x}^{*}_{\Ical^*} - \bar{\bm{y}}_{\Ical^*}\Vert_{2}^{2} 
			+\tfrac{\theta^2 -\theta^{2(k+1)}}{1-\theta^2}\Vert\beps_{\Ical^*}^{\bar{k}}\Vert_{1}^{2} +\tfrac{\theta^{2p} - (\theta^{2p})^{ k+1}}{1-\theta^{2p}}\Vert[\beps_{\Acal^*} ^{\bar{k} +1}]^p\Vert_{1}^{2}) \\ 
			\leq & C_\beta : = \kappa_2(\Vert\bm{x}^{\bar{k}-1}_{\Ical^*} - \bar{\bm{y}}_{\Ical^*}\Vert_{2}^{2} - \Vert\bm{x}^{*}_{\Ical^*} - \bar{\bm{y}}_{\Ical^*}\Vert_{2}^{2}+\tfrac{\theta^2}{1-\theta^2}\Vert\beps_{\Ical^*}^{\bar{k}}\Vert_{1}^{2} +\tfrac{\theta^{2p} }{1-\theta^{2p}}\Vert[\beps_{\Acal^*} ^{\bar{k} +1}]^p\Vert_{1}^{2}).
		\end{aligned}
	\end{equation*}
	
	\noindent Dividing both sides  by $k$  yields $\min_{\bar{k} \leq t \leq \bar{k} + k} \beta (\bm{x}^{t})   \leq  \sqrt{C_\beta / k}$.
\end{proof}
\subsection{Proof of Lemma \ref{Lem: approximation_is_controlled}}\label{app_D}
\begin{proof} 
	\noindent We consider two cases. If $0\le x_i \le \epsilon_i$ for $i \in [n]$, it holds that 
	\begin{equation}\label{bound.1.1} 
		(x_i +\epsilon_i)^p -x_i^p  <   (x_i +\epsilon_i)^p  \le     (2\epsilon_i)^p  = 
		2^p \epsilon_i^p, 
	\end{equation} 
	and 
	\begin{equation}\label{bound.1.2}
		\begin{split}
			& \ \big| |( \bar y_i - x_i )x_i  -  \lambda p  x_i^p| - |( \bar y_i - x_i )x_i  -  \lambda p  (x_i +\epsilon_i)^{p-1}x_{i}  | \big|\\
			\le &\  \lambda p x_i | (x_i +\epsilon_i)^{p-1} -x_i^{p-1} | \\
			\leq &   \ \lambda px_i \epsilon_i^{p-1}\leq \lambda p\epsilon_i^{p}.
		\end{split} 
	\end{equation}
	
	\noindent On the other hand, suppose $ \epsilon_i< x_i \le \bar y_i$, there then exists $\hat x_i \in [x_i, x_i + \epsilon_i]$ such that 
	\begin{equation}\label{bound.2.1}
		(x_i+\epsilon_i)^p -x_i^p =      p   \hat x_i^{p-1} \epsilon_i  \leq   p  x_i^{p-1} \epsilon_i 
		<   p \epsilon_i^{p-1} \epsilon_i  =     p \epsilon_i^p <     2^p \epsilon_i^p, 
	\end{equation}
	implying 
	\begin{equation}\label{bound.2.2}
		\begin{aligned}
			&\ \big| |( \bar y_i - x_i )x_i  -  \lambda p  (x_i )^p| - |( \bar y_i - x_i )x_i  -  \lambda p  (x_i \!+\!\epsilon_i)^{p-1}x_{i}  | \big|\\
			\le & \   \lambda p x_i | (x_i +\epsilon_i)^{p-1} \!-\!x_i^{p-1} | \\
			\le & \     \lambda p(1-p) x_i \hat{x_{i}}^{p-2}\epsilon_i \\
			\le & \ \lambda p(1-p) x_i^{p-1}\epsilon_{i} 
			\le    \lambda p \epsilon_i^{p}.  
		\end{aligned} 
	\end{equation}
	Combining these two cases and summing up both sides of \eqref{bound.1.1} and \eqref{bound.2.1}, we have 
	$
	|\beta(\bx) - \beta_{\beps}(\bx) | <  2^p \|\bm{\epsilon}\|_p^p.
	$
	Similarly, summing up both sides of \eqref{bound.1.2} and \eqref{bound.2.2}, we have 
	$
	|\alpha(\bx, \lambda)-\alpha_{\beps}(\bx)| \le \lambda p  \sum_{i=1}^n \epsilon_i^p =  \lambda p  \|\bm{\epsilon}\|_p^p.
	$
	This completes the proof.
\end{proof}
\subsection{Proof of Lemma \ref{alphabeta}}\label{app_H}
\begin{proof} 
	From the optimality conditions~\eqref{lambda.l1}, we have 
	\begin{equation}\label{lambda.bound.1}
		\begin{aligned}
			\lambda^k = & \frac{\sum\limits_{i\in\Ical^k} (\bar y_i - x_i^k ) }{ p \sum\limits_{i\in\Ical^k} (x_i^{k-1}+\epsilon_{i})^{p-1}}  \leq 
			\frac{\sum\limits_{i\in\Ical^k} \bar y_i}{  p\sum\limits_{i\in\Ical^{k}}(\bar{y}_{i}+\epsilon_{i})^{p-1}} \leq \frac{\Vert\bar{\bm{y}}\Vert_{1}}{  p\min\limits_{i\in\Ical^{k}}                                                        (\bar{y}_{i}+\epsilon_{i})^{p-1}} \\
			\le &  \frac{\Vert\bar{\bm{y}}\Vert_{1}}{  p\min\limits_{i\in[n]}(\bar{y}_{i}+\epsilon_{i})^{p-1}} \le 
			\frac{\Vert\bar{\bm{y}}\Vert_{1}}{p(\max\limits_{i\in[n]}(\bar{y}_{i} + \epsilon_{i}))^{p-1}}  \le \frac{\Vert\bar{\bm{y}}\Vert_{1}}{p(\Vert\bar{\bm{y}}\Vert_{\infty} + \Vert\bm{\epsilon}\Vert_{\infty})^{p-1}},
		\end{aligned}
	\end{equation}
	where the first inequality to the fourth inequality hold due to Lemma \ref{lemma.primal} (i) and $(\cdot)^{p-1}$ is a monotonically descreasing function on $\mathbb{R}_{++}$.

	Now we prove (ii).  Since for any $\bm{a}, \bm{b}, \bm{c} \in \mathbb{R}^n_+$, it is true that 
	\begin{equation}\label{abc}
		\sum_{i=1}^na_ib_ic_i  \le \|\bm{a}\|_2 \|\bm{b}\|_2\|\bm{c}\|_2. 
	\end{equation}
	It follows that 
	\begin{equation*}
		\begin{aligned}
			& \ \alpha_{\beps}(\bx^k,\lambda^k)\\
			= & \ \sum_{i=1}^n | ( \bar y_i - x_i^k)x_i^k -  \lambda^k p (x_i^k+\epsilon_i)^{p-1}x_i^k|\\
			= & \ \sum_{i=1}^n | ( \bar y_i - x_i^k)x_i^k - \lambda^k p (x_i^{k-1}+\epsilon_i)^{p-1}x_i^{k} + \lambda^k p x_i^k[ (x_i^{k-1}+\epsilon_i)^{p-1} \!-\!(x_i^k+\epsilon_i)^{p-1}]|\\
			= & \ \lambda^k p\sum_{i=1}^n   x_i^k | (x_i^{k-1}+\epsilon_i)^{p-1} -    (x_i^k+\epsilon_i)^{p-1}|\\
			\leq & \ \lambda^k p (1-p)\sum_{i=1}^n   x_i^k (\tilde x_i^{k-1}+\epsilon_i)^{p-2} | x_i^{k-1} - x_i^k|\\
			\le &  \ \lambda^k p (1-p)\sum_{i=1}^n  \bar y_i  \epsilon_i^{p-2} | x_i^{k-1} - x_i^k|\\
			\le & \frac{\Vert\bar{\bm{y}}\Vert_{1}\Vert\bar{\bm{y}}\Vert_{2}}{(\Vert\bar{\bm{y}}\Vert_{\infty} + \Vert\bm{\epsilon}\Vert_{\infty})^{p-1}}  (1-p) \|\beps^{p-2}\|_{2} \| \bx^{k-1} - \bx^k\|_2,
		\end{aligned}
	\end{equation*}
	where the third  equality is by the optimality conditions~\eqref{eq: subproblem_kkt_real}, the first  inequality is by Lagrange's mean value theorem with 
	$\tilde x_i^{k-1}$ between $x_i^k$ and $x_i^{k-1}$, the second  inequality is by Lemma \ref{lemma.primal} (i) and 
	the last inequality is by~\eqref{lambda.bound.1} and~\eqref{abc}. 
	
	As for (iii), we have from Lemma \ref{lemma.primal} (iii) that 
	\begin{equation*}
		\begin{aligned}
			\vert  \|\bm{x}^k+\bm{\epsilon}\|_p^p - \gamma\vert = &\  \left| \sum_{i=1}^{n} w_i^k(x^{k+1}_i - x_i^k) \right|    \leq \sum_{i=1}^{n}\vert  w_i^k(x^{k+1}_i - x_i^k)\vert \\
			\leq & \ 
			p \sum_{i=1}^{n} \vert\epsilon_i^{p-1}(x^{k+1}_i - x_i^k)\vert 
			\le p \|\beps^{p-1}\|_2 \|\bx^{k+1}-\bm{x}^k\|_2,
		\end{aligned}
	\end{equation*}
	where the second inequality is true because $(\cdot)^{p-1}$ is nonincreasing and the last inequality is by~\eqref{abc}. This completes the proof. 
\end{proof}
\subsection{Proof of Lemma \ref{Lemma: the maximum iteration}}\label{app_I}
\begin{proof}
	By Lemma \ref{alphabeta}, at the $k$th iteration (note that this is the last iteration running in the algorithm), it follows that
	\begin{equation}\label{eq: alpha_bound}
		\begin{aligned}
			(\frac{\varepsilon}{2})^{2} \le \ \alpha_{\beps}^{2}(\bm{x}^{k}, \lambda^{k}) &=\ \kappa_3^2 \Vert\bar{\bm{y}}\Vert_{2}^{2}  (1-p)^{2} \|\beps^{p-2}\|_{2}^{2} \| \bx^{k-1} - \bx^k\|_{2}^{2}\\
			&\le  \ \kappa_3^2 \Vert\bar{\bm{y}}\Vert_{2}^{2}  (1-p)^{2} \|\beps^{p-2}\|_{2}^{2}   (\|\bm{x}^{k-1}-\bar{\bm{y}}\|_2^2 - \|\bm{x}^{k}-\bar{\bm{y}}\|_2^2 ), \\
		\end{aligned}
	\end{equation}
	where the second inequality holds owing to~\eqref{eq: sum1}. Similarly,
	\begin{equation}\label{eq: beta_bound}
		\begin{aligned}
			(\frac{\varepsilon}{2})^{2}\le \ \beta_{\beps}^{2}(\bm{x}^{k})  &=\ p^{2} \|\beps^{p-1}\|_{2}^{2} \|\bx^{k+1}-\bm{x}^k\|_{2}^{2}\\
			&\le  \ p^{2} \|\beps^{p-1}\|_{2}^{2} (\| \bm{x}^k-\bar{\bm{y}} \|_2^2 - \| \bm{x}^{k+1} -\bar{\bm{y}} \|_2^2 ). \\
		\end{aligned}
	\end{equation}
	\noindent  Reindexing each term by $t$ and summing up both sides of \eqref{eq: alpha_bound} from $t=1$ to $k$,  we have
	\begin{equation*}
		\begin{aligned}
			(\frac{\varepsilon}{2})^{2}k \le  \ \sum_{t=1}^k \alpha_{\beps}^{2}(\bm{x}^{t},\lambda^{t}) &\leq \kappa_3^2 \Vert\bar{\bm{y}}\Vert_{2}^{2}  (1-p)^{2} \|\beps^{p-2}\|_{2}^{2}   (\sum_{t=1}^k\|\bm{x}^{k-1}-\bar{\bm{y}}\|_2^2 - \|\bm{x}^{k}-\bar{\bm{y}}\|_2^2 )\\
			&=\ \kappa_3^2 \Vert\bar{\bm{y}}\Vert_{2}^{2}   (1-p)^{2} \|\beps^{p-2}\|_{2}^{2}  (\|\bm{x}^0 - \bar{\bm{y}}\|_{2}^{2} - \|\bm{x}^{k} - \bar{\bm{y}}\|_2^2)\\
			& \le  \ \kappa_3^2 \Vert\bar{\bm{y}}\Vert_{2}^{2} (1-p)^{2} \|\beps^{p-2}\|_{2}^{2}  (\|\bm{x}^0 - \bar{\bm{y}}\|_{2}^{2} - \|\bm{x}^* - \bar{\bm{y}}\|_2^2).
		\end{aligned}
	\end{equation*}
	We then conclude 
	\begin{equation}\label{eq: bound4.1}
		k_{\alpha} \leq     4(1-p)^{2}  \kappa_3^2 \Vert\bar{\bm{y}}\Vert_{2}^{2}   \| \bm{\epsilon}^{p-2}\|_2^{2}(\|\bm{x}^0 - \bar{\bm{y}}\|_{2}^{2} - \|\bm{x}^* - \bar{\bm{y}}\|_2^2)   / \varepsilon^{2}.
	\end{equation}	
	Using similar proof techniques for $\beta(\bm{x}^{k})$, we have
	\begin{equation}\label{eq: bound4.2}
		k_{\beta} \leq   4p^{2} \|\beps^{p-1}\|_{2}^{2}  (\|\bm{x}^0 - \bar{\bm{y}}\|_{2}^{2} - \|\bm{x}^* - \bar{\bm{y}}\|_2^2) /  \varepsilon^{2}.
	\end{equation}
	Combining \eqref{eq: bound4.1} and \eqref{eq: bound4.2}, we therefore have
	\begin{equation*}
		k \leq \max(k_{\alpha}, k_{\beta}).
	\end{equation*}
	This completes the proof.
\end{proof}

\section{Root-searching procedure description}\label{app_RS} The implementation of RS follows the details suggested in \cite{chen2019outlier}. In particular, for solving these nonlinear equations, we directly call the built-in~\textit{optimize.newton} routine in the open-source software SciPy~\cite{virtanen2020scipy} and use its default values for all parameters (e.g., the maximum number of iterations).  We terminate the bisection method in the exterior loop whenever the maximum number of iterations (IterMax = $10^{3}$) is exceeded or  the interval becomes too narrow ($\lambda_{\textrm{high}} - \lambda_{\textrm{low}} < 10^{-10}$). In particular, we set initial $\lambda_{\textrm{low}} = 10^{-15}$ and $\lambda_{\textrm{high}} = \tfrac{\|\bar{\bm{y}}\|_{\infty}^{2-p}}{p(1-p)(\tfrac{1}{1-p} + 1)^{2-p}}$ suggested by~\cite{chen2019outlier}. Moreover, an approximate solution $\bm{x}$ is returned if $\frac{1}{n} | \|\bx\|_p^p - \gamma | < 10^{-8}$ holds.

\begin{algorithm}[H]
	\caption{Projection onto the $\ell_{p}$ ball using RS method}
	\label{alo: RS}
	\begin{algorithmic}[1]
		\STATE \textbf{Input}: $\bm{y}$, $p$, $\gamma$, \textrm{tol} and $\textrm{IterMax}$. Let $\bar{\bm{y}} = |\bm{y}|$. 
		\STATE \textbf{Initialization}: Choose $\lambda_{\textrm{low}} = 2.22\times10^{-16}$ and  $\lambda_{\textrm{high}} = \tfrac{(\bar{y}_{i}^{\textrm{max}})^{2-p}}{p(1-p)(\tfrac{1}{1-p} +1)^{2-p}}$. Set  $\bm{x} = \bm{0}_{n}$ and $k=0$.
		\WHILE{$\lambda_{\textrm{high}} - \lambda_{\textrm{low}} >= \textrm{tol}$~$\textrm{and}$~ $k <= \textrm{IterMax}$}
		\STATE $\lambda^{k} = (\lambda_{\textrm{low}} + \lambda_{\textrm{high}})/2$.
		\STATE (Solve $n$ nonlinear equations) Use Newton's method to solve
		\begin{equation}\label{eq: optimality}
			x_{i}^{k} - \bar{y}_{i} + p \lambda^{k}(x_{i}^{k})^{p-1} = 0,\quad\forall i\in[n]
		\end{equation}
		with the initial point $x_{\textrm{ini}}\sim U[\tfrac{2-2p}{2-p}\bar{y}_{i}, \bar{y}_{i}]$ is used in Newton's method.
		\IF{Newton's method fails to output a solution}
		\STATE $x_{i}^{k} = 0$
		\ENDIF
		\STATE Obtain $\bm{x}^{k} = [x_{1}^{k},~x_{2}^{k},~\ldots, ~x_{n}^{k}]^{\top}$.
		\IF {\eqref{eq: practical_termination} is met}
		\STATE $\lambda^{*} = \lambda^{k}$.
		\ELSIF {$\|\bm{x}^{k}\|_{p}^{p} > \gamma$}
		\STATE {$\lambda_{\textrm{low}} = \lambda^{k}$.}
		\ELSIF {$\|\bm{x}^{k}\|_{p}^{p} < \gamma$}
		\STATE {$\lambda_{\textrm{high}} = \lambda^{k}$.}
		\ENDIF
		\STATE $k \gets k+1$.
		\ENDWHILE
	\end{algorithmic}
\end{algorithm}

\bibliographystyle{ieeetr}
\bibliography{references}

\end{document}